\documentclass[aap,preprint]{imsart}
\usepackage{latexsym,amssymb,amsmath,amsfonts,amsthm,xcolor}
\usepackage{comment}
\usepackage{tikz}

\startlocaldefs
\newtheorem{theorem}{Theorem}

\newtheorem{lemma}{Lemma}

\def\beq{ \begin{equation} }
\def\eeq{ \end{equation} }
\def\mn{\medskip\noindent}
\def\ms{\medskip}

\def\ep{\epsilon}

\def\square{\vcenter{\vbox{\hrule height .4pt
			\hbox{\vrule width .4pt height 5pt \kern 5pt
				\vrule width .4pt} \hrule height .4pt}}}

\def\RR{\mathbb{R}}
\def\ZZ{\mathbb{Z}}

\def\clearp{}

\def\P{{\mathbb P}}     
\def\E{{\mathbb E}}     
\def\<{{\langle}} 
\def\>{{\rangle}} 
\def\1{{\bf 1}}         
\endlocaldefs

\begin{document}

\begin{frontmatter}
	
	\title{Genealogies in Expanding Populations}
	\runtitle{Genealogies in Expanding Populations}

	\begin{aug}
\author{\fnms{Rick} \snm{Durrett}\thanksref{t1}\ead[label=e1]{rtd@math.duke.edu}}
\and
\author{\fnms{Wai-Tong (Louis)} \snm{Fan}\thanksref{t2}\ead[label=e2]{ariesfanhk@gmail.com}}
		
	\thankstext{t1}{Partially supported by NSF grant DMS 1305997 from the probability program.}
	\thankstext{t2}{Partially supported by  Army Research Office W911NF-14-1-0331.}
		
		\runauthor{Rick Durrett and Wai-Tong (Louis) Fan}
		
		\affiliation{Duke University and University of North Carolina at Chapel Hill}
		
		\address{Department of Mathematics, Box 90320 \\
			Duke U., Durham, NC 27708-0320\\
			\printead{e1}\\
			}
		\address{Department of Statistics and Operations Research \\
			UNC, Chapel Hill, NC 27708-0320\\
			\printead{e2}\\
			}
	\end{aug}

	\begin{abstract}
The goal of this paper is to prove rigorous results for the behavior of genealogies in a one-dimensional long range biased voter model introduced by 
Hallatschek and Nelson \cite{HN08}. The first step, which is easily accomplished using results of Mueller and Tribe \cite{MT95}, 
is to show that when space and time are rescaled correctly, our biased voter model converges to a Wright-Fisher SPDE. A simple extension of 
a result of Durrett and Restrepo \cite{DurRep} then shows that the dual branching coalescing random walk 
converges to a branching Brownian motion in which particles
coalesce after an exponentially distributed amount of intersection local time. Brunet et al. \cite{BDMM} have conjectured that genealogies 
in models of this type are described by the Bolthausen-Sznitman coalescent, see \cite{NH13}. However, in the model we study there are no 
simultaneous coalescences. Our third and most significant result concerns ``tracer dynamics'' in which some of the initial particles 
in the biased voter model are
labeled. We show that the joint distribution of the labeled and unlabeled particles converges to 
the solution of a system of stochastic partial differential 
equations. A new duality equation that generalizes the one Shiga \cite{Shiga88} developed for the Wright-Fisher SPDE is the key to the proof of that result.
	\end{abstract}
	
	\begin{keyword}[class=MSC]
		\kwd[Primary ]{60K35, 60H15}
		\kwd[; secondary ]{92C50.}
	\end{keyword}
	
	\begin{keyword}
		\kwd{Biased voter model}
		\kwd{SPDE}
		\kwd{Duality}
		\kwd{Scaling limit}
		\kwd{Genealogies.}
	\end{keyword}
	
\end{frontmatter}

\section{Introduction}

Since the 1970s the dominant viewpoint in cancer modeling has been that tumors evolve through clonal succession, i.e., there is a sequence of mutations advantageous for cell growth that sweep to fixation. However, recent work of Sottoriva et al.~\cite{BigBang} has shown that in colon cancer most of common mutations in a tumor were present at initiation, while mutations that arise later are only present in small regions of the tumor. The long range goal of our research is to understand the behavior of genealogies in a growing tumor and the resulting patterns of genetic heterogeneity. This is important because cancer treatment may fail if one does not have an accurate knowledge of the mutational types present in the tumor. 

The genetic forces at work in a growing tumor are very similar to those in a population expanding into a new geographical area. In this case most of the advantageous mutation occur near the front, see Edmonds, Lille, and Cvalli-Sforza \cite{ELCS},
and simulation studies of a spatial model in \cite{KCE}. There are a half dozen papers by Oscar Hallatschek and collaborators \cite{HHRN07, HN08, HN10, Korolev, LHN12, NH13} using nonrigorous arguments to study genealogies in an expanding population. A rigorous analysis in dimensions $d \ge 2$ seems difficult, so we will restrict our attention here to a one dimensional model closely related to one introduced by Hallatshcek and Nelson \cite{HN08}. Our version of their model is a sequence of biased voter models on
$$
\Lambda_n:=(L_n^{-1}\ZZ) \times \{1, \ldots, M_n \}.
$$
There is one cell at each point of $\Lambda_n$, whose cell-type is either 1 or 0. The cells in deme $w \in L_n^{-1}\ZZ$ only interact with those in demes $w-L_n^{-1}$ and $w + L_n^{-1}$. Hence each cell $x=(w,i)$ has $2M_n$ neighbors. Type-0 cells reproduce at rate $2M_nr_n$, type-1 cell at rate $2M_n(r_n+\theta R_n^{-1})$ where $\theta\in[0,\infty)$. When reproduction occurs the offspring replaces a neighbor chosen uniformly at random. In the terminology of evolutionary games, this is birth-death updating.

\subsection{Interacting particle system duality}

Let $\xi_t(x)$ be the type of the cell at $x$ at time $t$. Our (rescaled) biased voter model  $(\xi_t)_{t\geq 0}$  can be constructed using  two independent families of  i.i.d.~Poisson processes: $\{P^{x,y}_t:\;x\sim y\}$ that have rate $r_n$, and $\{ \tilde P^{x,y}_t:\;x\sim y\}$, that have rate $\theta R_n^{-1}$. 
At a jump time of $P^{x,y}_t$, the cell at $x$ is replaced by an offspring of the one at $y$. At a jump time of $\tilde P^{x,y}_t$, the cell at $x$ is replaced by an offspring of the one at $y$ only if $y$ has cell-type 1. 
To avoid clutter we have suppressed the superscript $n$'s on $\xi_t$ and the two Poisson processes.

To formally construct the process and define a useful dual process, we recall the graphical representation of the biased voter model introduced by Harris (1976) and developed by Griffeath (1978). The ingredients are arrows that spread fluid in the direction of their orientation, and $\delta$'s which are dams that stop fluid.
At times $s$ of $\tilde P_{x,y}$ we draw an arrow from $(s,y) \to (s,x)$. At times $s$ of $P_{x,y}$ we draw an arrow from $(s,y) \to (s,x)$
and put a $\delta$ at $(s-,x)$. The $s-$ indicates that the $\delta$ is placed just before the head of the arrow. Intuitively, we inject fluid into the bottom of the graphical representation at the sites of the configuration that
are 1 and let it {\it flow up}.  A site $x$ is in state 1 at time $t$ if and only if it can be reached by fluid.
If there is fluid at $y$ at time $s$ and an arrow (with no $\delta$) from $(s,y) \to (s,x)$, the fluid will spread to $x$, i.e., there is a birth at $x$
if it is in state 0. If $x$ is already occupied no change occurs.

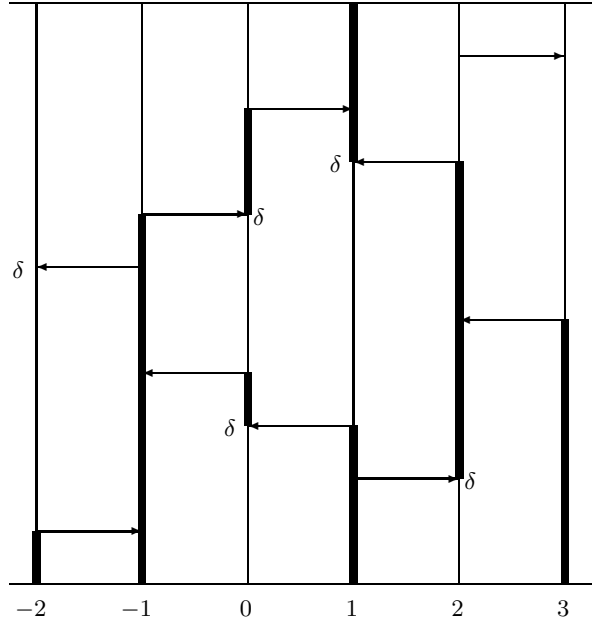
\begin{figure}[h]
\begin{center}
\begin{picture}(280,280)
\put(30,30){\line(1,0){220}}
\put(30,250){\line(1,0){220}}
\put(40,30){\line(0,1){220}}
\put(80,30){\line(0,1){220}}
\put(120,30){\line(0,1){220}}
\put(160,30){\line(0,1){220}}
\put(200,30){\line(0,1){220}}
\put(240,30){\line(0,1){220}}
\put(32,18){$-2$}
\put(72,18){$-1$}
\put(117,18){0}
\put(157,18){1}
\put(197,18){2}
\put(237,18){3}
\put(40,50){\vector(1,0){40}}
\put(160,70){\vector(1,0){40}}
\put(160,90){\vector(-1,0){40}}
\put(120,110){\vector(-1,0){40}}
\put(240,130){\vector(-1,0){40}}
\put(80,150){\vector(-1,0){40}}
\put(80,170){\vector(1,0){40}}
\put(200,190){\vector(-1,0){40}}
\put(120,210){\vector(1,0){40}}
\put(200,230){\vector(1,0){40}}
\put(151,186){$\delta$}
\put(122,166){$\delta$}
\put(31,146){$\delta$}
\put(111,86){$\delta$}
\put(202,66){$\delta$}
\linethickness{1.0mm}
\put(160,250){\line(0,-1){60}}
\put(200,190){\line(0,-1){120}}
\put(240,130){\line(0,-1){100}}
\put(160,90){\line(0,-1){60}}
\put(120,210){\line(0,-1){40}}
\put(80,170){\line(0,-1){140}}
\put(40,50){\line(0,-1){20}}
\put(120,110){\line(0,-1){20}}
\end{picture}
\end{center}
\caption{Duality for the biased voter model.}
\label{fig:dual}
\end{figure}

If there is an arrow-$\delta$ from $(s,y) \to (s,x)$, then a little thought reveals

\begin{center}
\begin{tabular}{lcclc}
\multicolumn{2}{c}{before} & \qquad & \multicolumn{2}{c}{after} \\
{\it y} & {\it x} && {\it y} & {\it x} \\
1 & 0 && 1 & 1 \\
0 & 1 && 0 & 0 \\
1 & 1 && 1 & 1 \\
0 & 0 && 0 & 0 
\end{tabular}
\end{center}

\noindent
In the first case the fluid spreads from $y$ to $x$ as if there was no $\delta$. In the second, there is no fluid to be spread but the dam stops
the fluid at $x$. In the third the dam stops the fluid at $x$, but it replaced by fluid from $y$. In the fourth, there is no 
fluid so nothing happens. Thus the effect in all cases is that $x$ imitates $y$.

To define the dual, we inject fluid at $x$ at time $t$ and let it {\it flow down}. It is again stopped by dams but now moves across
arrows in the direction OPPOSITE to their orientation. Given $z\in \Lambda_n=(L^{-1}_n\ZZ) \times \{1, \ldots, M_n \}$, we define the dual process $\zeta^{t,z}_s$, $0\le s \le t$ 
to be the set of sites at time $t-s$ that can be reached by fluid starting at $z$. A little thought shows that 
$\zeta^{t,z}_0=\{z\}$ and follows the following rules:

\begin{itemize} 
  \item If a particle in $\zeta^{t,z}_s$ is at $x$ and an arrival in $P^{x,y}$ occurs at time $t-s$ then the particle jumps to $y$.
  \item If a particle in $\zeta^{t,z}_s$ is at $x$ and an arrival in $\tilde P^{x,y}$ occurs at time $t-s$ then the particle gives birth to a new particle at $y$.
  \item If a jumping particle or an offspring lands on another particle in $\zeta^{t,z}_s$, then the two particles coalesce to 1.
\end{itemize} 

\noindent
For a picture see Figure \ref{fig:dual}. There the dark lines give the occupied sites in $\zeta^{t,1}_s$. At the end $\zeta^{t,1}_t = \{ -2, -1, 1, 3 \}$.

To extend the definition to a collection of sites $A$, we let $\zeta^{t,A}_s = \cup_{z\in A} \zeta^{t,z}_s$. 
We defined our dual process starting at a fixed time $t$ so that the relationship in Lemma \ref{duality}
holds with probability 1. If we have two times $t<t'$ then the distributions of  $\zeta^{t,A}_s$ and  $\zeta^{t',A}_s$ agree up to time $t$,
so the Kolmogorov extension theorem implies that we can define a process $\zeta^A_s$ {\it for all time} so that the distribution agrees with
$\zeta^{t,A}_s$ up to time $t$. 
The collection of processes $\{\zeta^A:\,A\subset \Lambda_n\}$ is referred to as the {\bf dual process of $\xi$}.

It follows from the definition that 
\begin{lemma} \label{duality}
$\xi_t(z)=1$ for some $z\in A$ if and only if $\xi_0(x) = 1$ for some $x \in \zeta^{t,A}_t$.
\end{lemma}

\subsection{Measure valued limit for the biased voter model}

We define the approximate density by
$$
u^n_t(w)  := \frac{1}{M_n} \sum_{i=1}^{M_n} \xi_t(w,i)
$$
and linearly interpolate between demes to define $u^n_t(w)$ for all $w\in\RR$. We retain the superscript $n$ on $u^n_t$ to be able to distinguish the approximating process from its limit. It is clear that for all $t\geq 0$, we have $0\leq u^n_t(w)\leq 1$ for all $w\in \RR$ and $u^n_t\in C_b(\RR)$,  the set of bounded continuous functions on $\RR$. If we equip $C_b(\RR)$ with the metric
\beq
\|f\| = \sum_{k=1}^{\infty} 2^{-k} \sup_{|x|\leq k}|f(x)|
\label{normC}
\eeq
i.e., uniform convergence on compact sets, then $C_b(\RR)$ is Polish and the paths $t\mapsto u^n_t$ are $C_b(\RR)$ valued and c\`adl\`ag. 

\begin{theorem}\label{mvLimit}
Suppose that as $n\to\infty$, the initial condition $u_0^n$ converges in $C_b(\RR)$ to $f_0$ and that:
\begin{enumerate}
	\item[(a)] $\alpha_n \equiv r_n M_n/L_n^2 \to \alpha \in (0,\infty)$
	\item[(b)] $\beta_n \equiv M_n/R_n \to \beta	\in[0,\infty)$
	\item[(c)] $\gamma_n \equiv r_n/L_n \to \gamma	\in[0,\infty)$ 
	\item[(d)] $L_n  \to \infty$ and $L_nR_n \to\infty$
\end{enumerate}
Then the approximate density process $(u_t^n)_{t\geq 0}$ converges in distribution in $D([0,\infty),\,C_b(\RR))$, as $n\to\infty$, to a continuous $C_b(\RR)$ valued process $(u_t)_{t\geq 0}$ which is the weak solution to the (stochastic) partial differential equation
\begin{equation}\label{uSPDE}
\partial_t u = \alpha\,\Delta u + 2\theta \,\beta\, u(1-u) + |4\gamma \,u(1-u)|^{1/2} \, \dot{W}
\end{equation}
with initial condition $u_0=f_0$. Here $ \dot{W}$ is the  space-time white noise on $[0,\infty)\times \RR$.
\end{theorem}

\noindent
Theorem \ref{mvLimit} is a straight forward generalization of a result of Mueller and Tribe \cite{MT95}. They considered a long range voter model on $\ZZ/n$ in which two voters are neighbors if $|x-y| \le n^{1/2}$. Their voters change their opinion at rate $O(n)$ and imitate the opinion of a neighbor chosen at random. More precisely,
for each of the $2n^{1/2}$ neighbors, they adopt the opinion of that neighbor at rate $n^{1/2}$ if it is 0
and at rate $n^{1/2} + \theta n^{-1/2}$ if it is 1. Their model corresponds roughly in our situation to $L_n = M_n = R_n = r_n = n^{1/2}$. Their limit is
$$
\partial_t u = \frac{1}{6} \,\Delta u + 2\theta \, u(1-u) + |4\,u(1-u)|^{1/2} \,  \dot{W}.
$$
$\beta=\gamma=1$ while the 1/6 comes from the fact that the variance of the uniform distribution on $[-1,1]$ is 1/3.

\medskip
{\it Why is Theorem \ref{mvLimit} true?} If this discussion becomes confusing the reader can skip it. In Section \ref{sec:AMP} the analysis of the approximate martingale problem will give a more rigorous version of this explanation. 

\begin{itemize}
  \item 
To begin we note that only branching events change the expected number of ones. They occur from each site at rate $2\theta M_n/R_n$. Birth events from $x$ to $y\sim x$ with $\xi_t(x)=1$ and $\xi_t(y)=0$ will change the local density of 1's by $1/M_nL_n$. In an interval of length $h$ there are $hL_nM_n$ possibilities for $x$,  so using the assumption $M_n/R_n \to \beta$ gives the drift term $\approx 2\theta \,\beta\, u(1-u)$. 

\item
Under our assumptions, branching events are much less frequent than voting events ($r_n/L_n \to \gamma$ and $L_nR_n\to\infty$ imply $r_n/R_n^{-1} \to\infty$) so for the rest of the computation argument they can be ignored. Voting events occur from each site at rate $2 r_nM_n$. If they go from $x$ to $y\sim x$ with $\xi_t(x) \neq \xi_t(y)$ then the local density of 1's will change by $\pm 1/M_nL_n$. In an interval of length $h$ there are $hL_nM_n$ possibilities for $x$, so using the assumption $r_n/L_n \to \gamma$ the infinitesimal variance term is $\approx 4\gamma u(1-u)$. 

\item
To see the source of the Laplacian note that in the dual particles jump at rate $2 r_n M_n$ by an amount that $\pm 1/L_n$ with equal probability
so the genealogies converge to Brownian motions with variance $2\alpha t$. Recalling the 1/2 in the generator of Brownian motion we have the term $\alpha\Delta u$. 
\end{itemize}

\subsection{Convergence of the dual to Branching Brownian motions}  

Durrett and Restrepo \cite{DurRep} have earlier considered a related coalsescing random walk (with no births) on $\ZZ$. We use their notation even though it conflicts slightly from ours. There are $M$ haploid individuals at each site and nearest neighbor migration (i.e., jumps $\pm 1$ with equal probability) occur at rate $\nu$. Theorem 2 in \cite{DurRep} states the following.

\begin{theorem} \label{DR}
Sample one individual from the colony at 0 and another one from the colony at $L$. Let $t_0$ be the time at which they coalesce.
If $L\to\infty$ and $M\nu/L \to \alpha\in (0,\infty)$ then $2t_0/(L^2/\nu)$ converges in distribution to $\ell_0^{-1}(\alpha\tau)$ where $\ell_0$ is the local time at 0 of a standard Brownian motion started from $1$ and $\tau$ is an independent mean 1 exponential random variable.
\end{theorem}

\noindent
{\it Why is this true?} Since migration occurs at rate $\nu$, it tales time of order $L^2/\nu$ to move a distance $O(L)$.
In this time the differencesw in the location of the two particles will visit each integer between 0 and $L$ of order
$L/\nu$ times. If we want coalescence to have a probability strictrly between 0 and 1, we need $M\nu/L \to \alpha\in (0,\infty)$.

\medskip
The proof of Theorem \ref{DR} generalizes easily to give:

\begin{theorem} \label{BCRWlim}	Suppose that as $n\to\infty$, the conditions on $r_n$, $R_n$, $M_n$ and $L_n$ in Theorem \ref{mvLimit} hold. Then the dual process $\zeta$ of $\xi$ converges in distribution to a limit in which particles move according to
Brownian motions on the real line $\RR$ with variance $2\alpha t$, branching occurs at rate $\theta\beta$ and two particles coalesce when the local time at 0 of the difference between their locations  exceeds $\alpha\tau/\gamma$, where $\tau$ is a mean one exponential random variable independent of the particle motions. If $\gamma=0$ there is no coalescence.
\end{theorem}

If we let $\chi^A_t$ denote the spatial locations of the particles in the limit process in Theorem \ref{BCRWlim}, then Lemma \ref{duality} and Theorem \ref{BCRWlim} imply
\beq
\E \prod_{x\in A} (1- u_t(x)) = \E \prod_{y \in \chi^A_t} (1-u_0(y)).
\label{mvduality}
\eeq

\noindent
Here $A$ can be a multi-set, e.g., $\{ a, a, a, b, b \}$. In this case we start three particles at $a$ and two at $b$ and duality says
$$
\E[(1-u_t(a))^3(1-u_t(b))^2] =  \E \prod_{y \in \chi^A_t} (1-u_0(y)).
$$

The duality relationship described in the last paragraph is not new. In 1986, Shiga and Uchiyama \cite{ShiUch}  introduced  it to study a collection of Wright-Fisher diffusions coupled by migrations. Shiga \cite{Shiga88} used it to show uniqueness in law for \eqref{uSPDE}. Doering, Mueller, and Smereka \cite{DMS} gave a simpler description and derivation of the dual. See also Hobson and Tribe \cite{HobTri}.

\subsection{Tracer Dynamics}

Hallatschek and Nelson \cite{HN08} take an interesting approach to studying the ancestral lineage of a particle. That is, the path within the dual $\zeta^{x,t}_s$ that gives the actual ancestor at time $t-s$ of the individual at $x$ at time $t$. Think of our expanding population as a fluid and inject a small amount of red liquid at time 0. The locations of the red fluid at time $t$ will identify the locations of their progeny. In particle terms, we will have states 0, 1, and $1^*$ where the * indicates being labeled by the tracer.  $\xi(x)$ is the
indicator of the particle at $x$ being in states 1 or $1^*$.
To construct the labeled process it is convenient to use the graphical representation. If there is an arrow-$\delta$ from
$y$ to $x$ then $x$ adopts the state of $y$. If there is an arrow from $y$ to $x$ and $y$ is in state 1 or $1^*$ then $x$ will adopt the state of $y$, but nothing happens if $y$ is in state 0. In fluid terms, the color of fluid at $y$ replaces that at $x$. 

The interacting particle system (IPS) dual process gives us the set of possible ancestors of the particle at $x$ at time $t$. To analyze our new process with labeled and unlabeled particles,  we assign an ordering to the particles in the IPS dual in such a way that the {\bf FIRST occupied site} in the list will be the ancestor. The rules are as follows:

\begin{itemize}
  \item If particle $i$ jumps and there is no coalescence, then its location changes but the order does not.
  \item If particle $i$ jumps and coalesces with $j$ then the particle with the higher index is removed from the dual. The surviving particle has its location updated. The remaining particles are reindexed.
   \item If the $i$ particle gives birth then the new particle is labeled $i$, while all particles with indices $j \ge i$ have their indices increased by 1.
\end{itemize}

\noindent
To explain the definition, we will work through the example drawn in Figure \ref{fig:dual}. The successive states are
\begin{align*}
&\{1\}, \quad \{0,1\}, \quad \{0,2\}, \quad \{-1,2\}, \quad \{-1, 3, 2 \},\\
&\{ 0, -1, 3, 2\},\quad \{ 1, -1, 3, 2 \}, \quad \{ 1, -1, 3 \}, \quad \{ 1, -2, -1, 3 \}
\end{align*}
To see this note that under our rules, at the point where the dual jumps from $\{1\} \to \{0,1\}$, if there is a particle at 0 it will give
birth and replace any particle at 1. The next two events involve voting, so the affected particles move but do not change their position in the ordering. The next event is an arrow from 3 to 2. A particle at 3 will replace one at 2, so 3 is inserted in the list before 2. The next novel event is the seventh transition when the particles at 2 and 1 coalesce: at that time we drop the lower ranked particle.  For example, the actual ancestor of the particle $(t,1)$ is $1$ if $\xi_0(1)=1$; the actual ancestor is $-2$ if $\xi_0(1)=0$ and $\xi_0(-2)=1$.
 
Neuhauser \cite{CN92} used a similar dual to show that in the multitype contact process, if the particle death rates are equal, then the one with the higher birth rate takes over the system. The idea  of assigning an ordering to the particles to trace their genealogies is quite similar to that in the lookdown constructions of Donnelly and Kurtz \cite{DK99a, DK99b}.  There each particle is initially assigned a level. Genealogies of any set of particles can be read off from the evolution of the levels. See also Etheridge and Kurtz \cite{EK14} for the lookdown constructions for a very general class of population models.

To do computations, we let $\eta_t(x)=1$ if the individual at $x$ at time $t$ in the biased voter model is labeled and 0 otherwise.  We only label type 1's, so $\eta_t(x) \le \xi_t(x)$. As noted above $\eta_t(x)$ can be computed using the version of the dual $\zeta^{t,x}_s$ in which the points are ordered. That is, $\eta_t(x)=1$ if and only if the first occupied site in the list is labeled. More formally, if under the ordering $\zeta^{t,x}_t = \{ y_1, y_2, \ldots y_K \}$ (note that $K$ and $y_1, \ldots y_K$ are random variables), then
\begin{align*}
P( \xi_t(x) = 0 ) & = E \left( \prod_{i=1}^K (1-\xi_0(y_i)) \right) \equiv EF(\zeta_t^x) \\
P( \eta_t(x) = 1 ) & = E \left( \sum_{j=1}^K \eta_0(y_j) \prod_{i=1}^{j-1} (1-\xi_0(y_i)) \right) \equiv EG(\zeta_t^x)
\end{align*}

In the same way one computes
\begin{align}
P( \xi_t(x_1) = 0, \ldots \xi_t(x_m) = 0, & \, \eta_t(x_{m+1}) = 1, \ldots ,\eta_t(x_n) = 1 ) 
\label{ldual}\\
& = E\left[ \prod_{i=1}^m F(\zeta_t^{x_i}) \prod_{i=m+1}^n G(\zeta_t^{x_i}) \right]
\nonumber
\end{align}
A standard argument shows that the probabilities just computed determine the distribution of $(\xi_t,\eta_t)$.

As before, we define the approximate density for the labeled particles by
$$
\ell^n_t(w)  := \frac{1}{M_n} \sum_{i=1}^{M_n} \eta_t(w,i)
$$
and linearly interpolate to obtain a function $\ell^n_t(w)$ for all $w\in \RR$. The following is the main result of this paper.

\begin{theorem} \label{mvLimit2}
	Suppose that as $n\to\infty$, the conditions on $r_n$, $R_n$, $M_n$ and $L_n$ in Theorem \ref{mvLimit} hold, and that the initial condition $(u^n_0,\ell_0^n)$ converges in $C_b(\RR) \times C_b(\RR)$ to $(f_0,g_0)$.
	Then the pair of approximate densities $(u_t^n,\,\ell^n_t)_{t\geq 0}$ converges in distribution in  $D([0,\infty),\,C_b(\RR)\times C_b(\RR))$ to a continuous $C_b(\RR)\times C_b(\RR)$ valued process $(u_t,\,\ell_t)_{t\geq 0}$ which is the weak solution to the coupled (stochastic) partial differential equations
\begin{align*}
\partial_t u & = \alpha\,\Delta u + 2\theta \,\beta\, u(1-u) + |4\gamma \,\ell(1-u)|^{1/2} \, \dot{W}^0 + |4\gamma \,(u-\ell)(1-u)|^{1/2} \,\dot{W}^1\\
\partial_t \ell & = \alpha\,\Delta \ell + 2\theta \,\beta\, \ell\,(1-u)
+|4\gamma \,\ell(1-u)|^{1/2}\,\dot{W}^0 + |4\gamma \,\ell\,(u-\ell)|^{1/2}\,\dot{W}^2
\end{align*}
with initial condition $(u_0,\,\ell_0)=(f_0,\,g_0)$, where $\dot{W}^i$, $i=0,1,2$ are three independent space-time white noises on $[0,\infty)\times \RR$.
\end{theorem}

\noindent
As the proof shows the three noises $\dot{W}^0$, $\dot{W}^1$ and $\dot{W}^2$ refer to voting interactions between $1^*$ and 0, $1$ and $0$, and $1^*$ and $1$ respectively. The drift in $\ell_t$ is $2\theta\beta \,\ell (1-u)$ because labeled particles only have a selective advantage in competition with those of type 0.

We will prove our result by showing that the sequence of approximating processes is tight. To conclude that there is weak convergence,
we need to show 

\begin{lemma} \label{SPDEunique}
When $\gamma>0$, the solution of the coupled SPDE in Theorem \ref{mvLimit2} is unique in law.
\end{lemma}

\noindent
To do this we use our duality function \eqref{ldual}. Details are in Section \ref{sec:unique}.

In Theorems \ref{mvLimit}, \ref{BCRWlim} and \ref{mvLimit2}, the assumption $\alpha>0$ is used crucially in the proof of tightness, but we allow $\beta$ or $\gamma$ to be 0. The deterministic regime ($\gamma=0$) occurs, for instance, if $r_n=n^{1/a}$, $L_n= n^{1/b}$, $M_n=\lceil\alpha n^{2/b-1/a}\rceil$ and $R_n=M_n/\beta$, where $2a>b>a>0$. When $\gamma=0$  the limiting process is a PDE
\begin{align*}
\partial_t u & = \alpha\,\Delta u + 2\theta \,\beta\, u(1-u) \\
\partial_t \ell & = \alpha\,\Delta \ell + 2\theta \,\beta\, \ell\,(1-u).
\end{align*}
This PDE obviously has a unique weak solution (solve the first equation and then solve the second), but if one wants, this can be proved using duality.

\subsection{Lineage dynamics}

Using tracer dynamics Hallatschek and Nelson \cite{HN08}, see page 163 and Appendix A,  heuristically derived the probability density
$G(x,t|x',t')$ that an individual at $x'$ at time $t'$ is descended from an ancestor at $x$ at time $t<t'$. 
Here we describe their argument with the hope that someone can find a way to make it rigorous. 
Making sense of equation \eqref{HNG} for the ancestral lineages seems difficult, but since these lineages are embedded in the limit of the
branching coalescing random walk their behavior cannot be too pathological.

The authors of \cite{HN08} considered a more general situation in which the population density in a frame moving with velocity $v$ is
$$
\partial_t u(x,t) = D \partial^2_x u(x,t) + v \partial_x u(x,t) + K(x,t)
$$
where, for instance, $K(x,t) = su(1-u) + \ep\sqrt{u(1-u)} Z$ with
$Z$ being a space-time White noise, and we have changed their $c$ to $u$. They concluded, see their (3), that
\begin{align*}
\partial_t G(x,t|x',t') & = -\partial_x J(x,t|x',t') \\
J(x,t|x',t') & = - D\partial_x G + \{ v + 2D\partial_x \log[u(x,t)] \} G.
\end{align*}
Here $(x',t')$ is thought of as the initial condition and $(x,t)$ as the finial condition, so this is the forward
equation
\beq
\partial_t G = D\partial^2_x G - \partial_x (\{ v + 2D\partial_x \log[u(x,t)] \} G)
\label{HNG}
\eeq
and the drift in the diffusion process is $ v + 2D\partial_x \log[c(x,t)]$.

As we will now show, a closely related equation ``follows'' from Theorem \ref{mvLimit2}. We used quotation marks because we use the nonexistent It\^o's formula 
for  the SPDE in Theorem \ref{mvLimit2} and apply it to the function $f(\ell,u) = \ell/u$ which is not continuous at $(0,0)$.

Suppose $(u,\,\ell)$ solves the SPDE and let $\rho=\ell/u$. By Theorem \ref{mvLimit2},
\begin{align*}
\partial_t \rho & =\frac{u\,\partial_t\ell \,-\,\ell\,\partial_tu}{u^2} 
=\, \frac{1}{u}\,\Big[\, \alpha\,\Delta \ell +  2\theta \,\beta\,\ell\,(1-u) + |4\gamma \ell (1-u)|^{1/2}\,\dot{W}^0 \\
& + |4\gamma \ell (u-\ell)|^{1/2}\,\dot{W}^2\,\Big]
\,- \frac{\ell}{u^2}\,\Big[ \, \alpha\,\Delta u + 2\theta \,\beta\, u(1-u) \\
& + |4\gamma \ell(1-u)|^{1/2} \,\dot{W}^0
+  |4\gamma (u-\ell)(1-u)|^{1/2} \,\dot{W}^1\,\Big].
\end{align*}

The terms involving $\theta\beta$ cancel. To combine the Laplacian terms we use the formula 
$$
\Delta\left(\frac{\ell}{u}\right) = \frac{\Delta \ell}{u}- \frac{\ell\,\Delta u}{u^2} 
- 2\,\frac{\partial_x u}{u}\cdot\partial_x\left(\frac{\ell}{u}\right).
$$
To add up the noises we note that since the $W^i$ are independent, the variances add up to
\begin{align*}
&\frac{(u-\ell)^2}{u^4} \ell(1-u) + \frac{\ell(u-\ell)}{u^2} + \frac{\ell^4(u-\ell)(1-u)}{u^4} \\
=&\, \frac{\ell(u-\ell)}{u^2} \left[ \frac{u-\ell)(1-u)}{u^2} + \frac{u^2}{u^2} + \frac{\ell(1-u)}{u^2} \right] \;=\; \frac{\ell(u-\ell)}{u^3}
\end{align*}
since $(u-\ell)(1-u) + \ell(1-u) + u^2 = u(1-u) + u^2 = u$. Combining our calculations,
\begin{equation}\label{rhoSPDE}
\partial_t \rho \,=\,\alpha\,\Delta \rho +2\alpha\,\partial_x \log u\cdot\partial_x \rho\,+\,|4\gamma \,\rho\,(1-\rho)/u|^{1/2}\,\dot{W}
\end{equation}
for some white noise $\dot{W}$.
To compare \eqref{rhoSPDE} with \eqref{HNG}, note that their $D=\alpha$, they work in a moving reference frame 
and their equation is for a fixed realization of the total population size; while ours is in a fixed reference frame
and does not condition on $u(t,x)$ and hence retains the 
fluctuation term $|4\gamma \,\rho\,(1-\rho)/u|^{1/2}\,\dot{W}$.

Equations \eqref{HNG} and \eqref{rhoSPDE} both contain drift terms of the form $\partial_x \log[u(x,t)]$. 
It is a well-known fact that solutions of the SPDE in
\eqref{uSPDE} are H\"older continuous with exponent $1/2-\ep$ in space and $1/4-\ep$ in time, so it is not clear how to make
sense of these equations. The fact that $\eta_t(x) \le \xi_t(x)$ means that $\ell \le u$, so in computing $\ell/u$ we will
never divide a positive number by 0. However solutions to $u$ have compact support \cite{MS95}, so it is not clear if the ratio of densities
$\ell^n/u^n$ will be tight.

\subsection{Concluding remarks} 

Theorem 1 is a special case of Theorem 4. Theorem \ref{BCRWlim} is proved in Section 2. The rest of the paper is devoted to the proof for Theorem \ref{mvLimit2}. In Section 3 we introduce stochastic integral representations for $(u^n_t,\ell^n_t)$ and formulate approximate martingale problems. This is now a common approach in the study of scaling limits of particle systems, see \cite{DurPer, CDP00, DMP}. The calculations for the $u$ equation are almost identical to those in Section 3 of Mueller and Tribe \cite{MT95}, but some minor changes are needed to study the joint distribution $(\ell,u)$. 

In Sections 5--7 we prove tightness.
Again many of the ideas come from \cite{MT95}, but since they only write out the details for their contact process limit theorem, and we have to prove the joint convergence, we have written out the details. See Kleim \cite{Kleim} for a recent example of convergence of rescaled Lotka-Volterra models to a one-dimensional SPDE, this time with a cubic drift term.  The main ideas of the tightness proof are given in Section 5. Two lemmas that require a lot of computation are proved in Section 6. Some nonstandard random walk estimates are proved in Section 7. Finally Lemma \ref{SPDEunique}, which establishes distributional uniqueness for the coupled SPDE by using a duality based on \eqref{ldual}, is proved in Section 8. 

\bigskip
 
{\bf Acknowledgement.}  We would like to thank Carl Mueller and Edwin Perkins for their help while we were writing this paper. Two referees made many detailed comments that improved the readability (and correctness) of this paper. 

\clearp

\section{Proof of Theorem \ref{BCRWlim}} \label{sec:PfTh3}

To get started suppose that there are no births in the dual and consider the special case in which there are initially two particles. Let $S^n_t$ be a random walk that jumps from $x$ to $x\pm 1/L_n$ at rate $2r_nM_n$. The reader should think of this difference of the location of two particles (hence the factor of $2$ in $2r_nM_n$), but we allow the two lineages to move independently even after they hit. An elementary computation shows that if
$$
V^n_t = \frac{4r_nM_n}{L_n} \int_0^t 1_{( S^n_s = 0)} \, ds
$$
$|S^n_t|-V^n_t$ is martingale. As $t\to\infty$, $|S^n_t|$ converges to the absolute value of a Brownian motion $B_t$ with variance $4\alpha t$.

The next result has been proved for finite variance random walks by Borodin \cite{Borodin} in 1981. To keep this paper self-contained, we will give a simple proof for the nearest neighbor case.

\begin{lemma}
$V^n_t$ converges in distribution to $\ell_0(t)$ the local time at 0 for the Brownian motion $B_t$.
\end{lemma}

\begin{proof}
An easy computation for simple random walk shows that for any stopping time
$$
EV^n_{S+t} - V_{S} \le E_0V^n_t \le c \sqrt{t}
$$
so by Aldous' criterion (see e.g., Theorem 4.5 on page 320 of Jacod and Shiryaev \cite{JaShir} the sequence
$V^n_t$ is tight. Let $V^{n(k)}_t$ be a convergent subsequence with limit $V_t$.
$|S^{n(k)}_t - V^{n(k)}_t$ is a martingale. Using the maximal inequality on the random
walk, and the dominated convergence theorem on the increasing process, both processes
converge to their limits in $L^1$. Since conditional expectation is a contraction in
$L^1$, it follows that $|B_t| - V_t$ is a martingale. Now $\ell_0(s)$ is the increasing process associated with
$|B_t|$ See e.g., (11.2) on page 84 of Durrett's Stochastic Calculus book \cite{DSC}. By the uniqueness of the Doob-
Meyer decomposition, there is only one subsequential
limit, so the entire sequence converges in distribution to $\ell_0(t)$.
\end{proof}

The sojurn times at 0 are independent and exponential with rate $4r_nM_n$, so the number of visits to 0 up to time $t$
$$
N^n_t \sim  4r_nM_n \int_0^t 1_{( S^n_s = 0)} \, ds  
$$
and it follows that $N^n_t/L_n \to \ell_0(t)$. 
On each visit to 0, the two particles have a probability $1/M_n$ to coalesce. Our assumptions imply that
$$
\frac{M_n}{L_n} \cdot \frac{r_n}{L_n} \to \alpha \quad\hbox{so}\quad  \frac{L_n}{M_n}  \to \frac{\gamma}{\alpha} 
$$
and the desired result has been proved for the processes with no births. To add births we note there are $M_n$ cells at each deme and $\tilde{P}_{xy}$ has rate $\theta R_n^{-1}$, so the branching rate at each deme is $\theta M_n\,R_n^{-1}\to \theta \beta$. 

\clearp

\section{Approximate Martingale Problems} \label{sec:AMP}

For simplicity we drop the subscript $n$'s on $M$, $L$, $R$ and $r$. We leave the superscript $n$ in $u^n_t$ and $\ell^n_t$
to distinguish the approximating processes from their limits. We write 
$$
\<f,g\>:= \frac{1}{L}\sum_{w\in L^{-1}\ZZ} f(w)g(w)
$$  
whenever $\< |f|, |g| \> < \infty$  and adopt the convention that $\phi(x):=\phi(w)$ when $x=(w,i)$.

\subsection{Type 1 particles}

To develop the martingale problem, we note that dynamics of $(\xi_t)_{t\geq 0}$ can be described by the equation
\begin{align}\label{Peq}
\xi_t(x) = \xi_0(x)  &+ \sum_{y\sim x} \int_0^t (\xi_{s-}(y) - \xi_{s-}(x)) \, dP^{x,y}_s \\
& + \sum_{y\sim x} \int_0^t \xi_{s-}(y)(1 - \xi_{s-}(x)) \, d\tilde P^{x,y}_s.
\nonumber
\end{align}
In the first integral, if $\xi_{s-}(y) =\xi_{s-}(x)$ then nothing happens.
\begin{align*}
&\hbox{If $\xi_{s-}(y) =1$ and $\xi_{s-}(x) =0$ then $\xi_s(x)=1$;} \\
&\hbox{if $\xi_{s-}(y) =0$ and $\xi_{s-}(x) =1$ then $\xi_s(x)=0$.}
\end{align*}
In the second integral, nothing happens unless $\xi_{s-}(y) =1$ and $\xi_{s-}(x) =0$. In this case $\xi_s(x)=1$.

Let $\phi: [0,\infty)\times \RR \to \RR$ be continuously differentiable in $t$, twice continuously differentiable in $x$
and have compact support in $[0,T] \times \RR$ for any $T$.
Applying integration by parts to $\xi_t(x) \phi_t(x)$, using \eqref{Peq}, and summing over $x$, we obtain for all $t\in[0.T]$,
\begin{align}
\<u^n_t,\,\phi_t\>- & \<u^n_0,\,\phi_0\>  -\int_0^t \<u^n_s,\partial_s\phi_s\>ds  
\nonumber\\
&= (ML)^{-1} \sum_x \sum_{y\sim x} \int_0^t (\xi_{s-}(y) - \xi_{s-}(x)) \phi_s(x) \, dP^{x,y}_s
\label{term1}\\
& + (ML)^{-1} \sum_x \sum_{y\sim x} \int_0^t \xi_{s-}(y)(1 - \xi_{s-}(x)) \phi_s(x) \, d\tilde P^{x,y}_s.
\label{term2}
\end{align}

{\bf Drift term. } We break \eqref{term2} into an average term and a fluctuation term	
\begin{align}
& (ML)^{-1} \sum_x \sum_{y\sim x} \int_0^t \xi_{s-}(y)(1 - \xi_{s-}(x)) \phi_s(x) \,\theta R^{-1} \, ds
\label{term2a} \\
& + (ML)^{-1} \sum_x \sum_{y\sim x} \int_0^t \xi_{s-}(y)(1 - \xi_{s-}(x)) \phi_s(x) \, (d\tilde P^{x,y}_s - \theta R^{-1} \, ds ).
\label{term2b}
\end{align}
Recalling the definition of the density, \eqref{term2a} becomes 
\begin{align*}
& \theta \cdot \frac{M}{R} \cdot \frac{1}{L} \sum_{w\in L^{-1}\ZZ} \int_0^t [u^n_{s-}(w-L^{-1}) + u^n_{s-}(w+L^{-1})] (1-u^n_{s-}(w)) \phi_s(w) \, ds
\end{align*}
Since $M/R\to\beta$, this converges to  
$$
\theta \beta \int_0^t \int_{\RR} 2 u_s(w) (1-u_s(w)) \phi_s(w) \, dw \, ds\,
$$ 
as $n \to\infty$.
Here, and in what follows, the claimed convergences follow once we have proved $C$-tightness (see Section 5 for the proof). We have established convergence
of finite dimensional distributions so the sequences of processes converge in distribution. 

The second term \eqref{term2b} is a martingale $E^{(2)}_t(\phi)$ with
\begin{align*}
\langle E^{(2)}(\phi) \rangle_t \le &\,\frac{\theta}{ R (ML)^{2}} \sum_x \sum_{y\sim x} \int_0^t \phi^2_s(x) \, ds\\
\leq &\, \frac{2\theta}{L R} \int_0^t \<1,\,\phi^2_s\>\,ds \to 0 \qquad\hbox{since $LR\to\infty$.}
\end{align*}

{\bf White noises. } We can rewrite \eqref{term1} as
\beq
(ML)^{-1} \sum_x \sum_{y\sim x} \int_0^t \{\xi_{s-}(y)[1-\xi_{s-}(x)] - \xi_{s-}(x)[1-\xi_{s-}(y)]\} \phi_s(x) \, dP^{x,y}_s.
\label{term1new}
\eeq
We now rewrite the integrand as 
\begin{align}
& \xi_{s-}(y)[1-\xi_{s-}(x)] \phi_s(y) - \xi_{s-}(x)[1-\xi_{s-}(y)] \phi_s(x) \label{ig1} \\
& + \xi_{s-}(y)[1-\xi_{s-}(x)] (\phi_s(x) - \phi_s(y)). \label{ig2}
\end{align}
We work first with \eqref{ig1}. 		
Interchanging the roles of $x$ and $y$ in the double sum, letting $\xi^c_t(z) = 1- \xi_t(z)$,
and writing $Q^{x,y}_t = P^{y,x}_s - P^{x,y}_s$ this part of \eqref{term1new} becomes
\beq 
(ML)^{-1} \sum_x \sum_{y\sim x} \int_0^t \xi^c_{s-}(y) \xi_{s-}(x) \phi_s(x) \, dQ^{x,y}_t.
\label{term1b}
\eeq

To prepare for treating the joint SPDE we split \eqref{term1b} into $Z^0_t(\phi)+Z^1_t(\phi)$ 
where the $Z^i_t(\phi)$ are defined by the next two equations
\begin{align}
&(ML)^{-1} \sum_x \sum_{y\sim x} \int_0^t  \xi^c_{s-}(y)\, \eta_{s-}(x) \phi_s(x) \, dQ^{y,x}_s, 
\label{noise0}\\
&(ML)^{-1} \sum_x \sum_{y\sim x} \int_0^t \xi^c_{s-}(y) \,(\xi_{s-}(x) - \eta_{s-}(x)) \phi_s(x) \, dQ^{y,x}_s.
\label{noise1}
\end{align}
These two martingale terms use the same Poisson processes but the product of their integrands is 0 
since $(1-\xi_t(x))\eta_t(x)$ vanishes for all $x$ and $t$. Hence these martingales are uncorrelated.

The variance process $\langle Q^{x,y} \rangle_t = 2r t$. 
Our assumptions on $\phi$ imply that  $\sup_{s\in[0,t]}|\phi_s(x)-\phi_s(y)| \le K_t/L$ for some $K_t<\infty$, so
ignoring the difference between $\phi_s(x)$ and $\phi_s(y)$, we have
$$
\langle Z^0(\phi) \rangle_t  =  4rL^{-2} \int_0^t \sum_{w \in L^{-1}\ZZ}  \ell^n_{s-}(w) (1 - u^n_{s-}(w)) \phi_s(w)^2 \, ds + o(1)
$$
which converges to  
$$ 
4\gamma \int_0^t  \int_\RR  \ell_s(w)(1-u_s(w)) \phi_s(w)^2 \, dw \, ds\quad\text{since }r/L \to \gamma.
$$
Similarly, $\langle Z^1(\phi) \rangle_t $ converges to
$$ 
4\gamma \int_0^t  \int_\RR  (u_s(w)-\ell_s(w))\,(1-u_s(w)) \phi_s(w)^2 \, dw \, ds.
$$

{\bf Laplacian Term. } We denote the discrete gradient and the discrete Laplacian respectively by
\begin{align}
\nabla_Lf(w)& :=L\,\Big(f(w+L^{-1})-f(w)\Big)\label{dnabla}\\
\Delta_Lf(w)&:= L^2\,\Big(f(w+L^{-1})+f(w-L^{-1})-2f(w)\Big). \label{dLaplacian}
\end{align}	
We break \eqref{ig2} into an average term and a fluctuation term
\begin{align}
& (ML)^{-1} \sum_x \sum_{y\sim x} \int_0^t \xi_{s-}(y) \xi^c_{s-}(x) [\phi_s(x) - \phi_s(y)] \, r \, ds
\label{term3a}\\
& + (ML)^{-1} \sum_x \sum_{y\sim x} \int_0^t \xi_{s-}(y) \xi^c_{s-}(x) [\phi_s(y) - \phi_s(x)]  ( dP^{x,y}_s - r \, ds).
\label{term3b}
\end{align}
We can replace $\xi^c_{s-}$ by 1 in \eqref{term3a} without changing its value, 
because by symmetry, $\sum_x \sum_{y\sim x}\xi_{s-}(y) \xi_{s-}(x) [\phi_s(x) - \phi_s(y)] =0$ for all $s>0$.

Doing the double sum over $y$ and then over $x\sim y$ the
above is
$$
\frac{rM}{L^2} \cdot \frac{1}{L} \sum_w u^n_{s-}(w) \Delta_L \phi_s(w)=\frac{rM}{L^2}\<u^n_{s-},\,\Delta_L \phi_s\>.
$$
By assumption $rM/L^2 \to \alpha$, so this term converges to $\alpha\int_{\RR} u_{s}\, \Delta\phi_s$. The other term, 
\eqref{term3b}, is a martingale $E^{(1)}_t(\phi)$ with
\begin{align}
\langle E^{(1)}(\phi) \rangle_t	\leq &\, \frac{r}{(ML)^2} \sum_x \sum_{y\sim x} \int_0^t  (\phi_s(x) - \phi_s(y))^2  \, ds\\
= &\, \frac{2r}{L^3} \int_0^t \<1,\, |\nabla_L\phi_s|^2\>  \, ds \to 0
\end{align}
since $r/L \to \gamma$ and $L \to \infty$. 

Combining our calculations, we see that in the limit $n\to\infty$,
\begin{align}
 \int_{\RR}u_t(w)&\,\phi_t(w)\,dw\, -   u_0(w)\,\phi_0(w)\,dw   - \int_0^t \int_{\RR} u_s(w)\,\partial_s\phi_s(w)\,dw\,ds  
\nonumber \\
&- \int_0^t \int_{\RR} \alpha u_s(w) \Delta \phi_s(w) -  2\theta \beta \, u_s(w)(1-u_s(w))\, \phi_s(w)\,dw \, ds
\end{align}
is a martingale with quadratic variation
\beq
4\gamma \int_0^t \int_{\RR} u_s(w)(1-u_s(w)) \phi^2_s(w) \,dw \,ds
\eeq
which is the martingale problem formulation of \eqref{uSPDE}.

\subsection{Labeled Particles}

To begin, we note that
in terms of the previously defined Poisson processes,
\begin{align}
\label{Peq2}
& \eta_t(x) - \eta_0(x) 
 =\,  \sum_{y\sim x} \int_0^t (\eta_{s-}(y) - \eta_{s-}(x)) \, dP^{x,y}_s \\
&+\sum_{y\sim x} \int_0^t \eta_{s-}(y) (1-\eta_{s-}(x))
- \xi_{s-}(y)(1-\eta_{s-}(y)) \eta_{s-}(x) \, d\tilde P^{x,y}_s.
\nonumber
\end{align}
The first term gives the voter interactions. For the second term, note that if $y$ is in state $1^*$ and $x$ is not, the number of $1^*$'s will
increase by 1, while if $x$ is in state $1^*$ and $y$ is in state 1 ($\xi_{s-}(y) =1$ and $\eta_{s-}(y)=0$), the number will decrease by 1. 

Arguing as for the type 1 particles while using \eqref{Peq2} instead of \eqref{Peq}, we get
\begin{align}
\langle \ell^n_t, \phi_t \rangle &- \langle \ell^n_0 ,\phi_0 \rangle  - \int_0^t \langle \ell^n_s, \partial_s \phi_s \rangle 
\nonumber\\
& = (ML)^{-1} \sum_x  \sum_{y\sim x} \int_0^t (\eta_{s-}(y) - \eta_{s-}(x)) \phi_s(x) \, dP^{x,y}_s 
\label{lterm1}\\
&+ (ML)^{-1} \sum_x \sum_{y\sim x} \int_0^t [\eta_{s-}(y) 
- \xi_{s-}(y) \eta_{s-}(x)] \phi_s(x) \, d\tilde P^{x,y}_s, \label{lterm2}
\end{align}
where we have simplified the second term of \eqref{Peq2} using $\xi_{s-}(y)\eta_{s-}(y)=\eta_{s-}(y)$.

{\bf Drift term. }
Breaking the second term \eqref{lterm2} into an average term and a fluctuation term as before, we conclude that 
as $n \to\infty$, the average term is 
\begin{align*}
\frac{\theta M}{L R} \sum_{w\in L^{-1}\ZZ} \int_0^t \big[ & \ell^n_{s}(w+L^{-1})-u^n_s(w+L^{-1})\ell^n_s(w) \\
& + \ell^n_{s}(w-L^{-1})-u^n_s(w-L^{-1})\ell^n_s(w)\big]\phi_s(w) \, ds
\end{align*}
which converges to
$$
\theta \beta \int_0^t \int_{\RR} 2 \ell_s(w) (1-u_s(w)) \phi_s(w) \, dw \, ds.
$$ 

{\bf White noises. } We again change the integrand in \eqref{lterm1} to 
$$
\{\eta_{s-}(y)[1-\eta_{s-}(x)] - \eta_{s-}(x) [1-\eta_{s-}(y)] \}\, \phi_s(x)
$$
and then split it into two parts as in \eqref{ig1} and  \eqref{ig2}. That is,
we rewrite the integrand as 
\begin{align}
& \eta_{s-}(y)[1-\eta_{s-}(x)] \eta_s(y) - \eta_{s-}(x)[1-\eta_{s-}(y)] \phi_s(x) \label{lig1} \\
& + \eta_{s-}(y)[1-\eta_{s-}(x)] (\phi_s(x) - \phi_s(y)). \label{lig2}
\end{align}

Arguing as before, we obtain the following sum coming from \eqref{lig1}.
\begin{align}
& (ML)^{-1} \sum_x \sum_{y\sim x} \int_0^t [1-\xi_{s-}(y)] \eta_{s-}(x) \phi_s(x) \, ( dP^{y,x}_s - dP^{x,y}_s ) 
\label{ltwoNoises}
\\
+ & (ML)^{-1} \sum_x \sum_{y\sim x} \int_0^t [\xi_{s-}(y) - \eta_{s-}(y)] \eta_{s-}(x)  \phi_s(x) \, ( dP^{y,x}_s - dP^{x,y}_s ).
\nonumber
\end{align} 
The first noise is the same as $Z^0_t(\phi)$ in \eqref{noise0} while the second noise, denoted by $Z^2_t(\phi)$, has variance converging to
$$ 
4\gamma \int_0^t  \int_\RR \ell_s(w)\, (u_s(w)-\ell_s(w))\, \phi_s(w)^2 \, dw \, ds.
$$
The product of any two of the three integrands in $Z^{i}_t(\phi)$ ($i=0,1,2$) is 0, so
these three martingales are uncorrelated. 

{\bf Laplacian term. } Breaking \eqref{lig2} into an average term and a fluctuation term as was done for \eqref{ig2}, we see that as $n\to\infty$, the average term 
$$
\frac{rM}{L^2}\<\ell^n_{s-},\Delta_L \phi_s\> \to \alpha\int_{\RR} \ell_{s}\, \Delta\phi_s.
$$

Combining our calculations, we see that in the limit, for any $\phi,\,\psi\in C^{1,2}_c([0,\infty)\times \RR)$,
	\begin{align}
	& \, \int_{\RR}(u_t\,\phi_t- u_0\,\phi_0 +  \ell_t\,\psi_t -\ell_0\,\psi_0)\,dw \notag\\
	& - \alpha \int_0^t  \int_{\RR}u_s\,(\partial_s\phi_s + \Delta\phi_s) + \ell_s\,(\partial_s\psi_s+\Delta\psi_s)\,dw \,ds \notag\\
	& - 2\,\theta \,\beta\int_0^t \int_{\RR} u_s\,(1-u_s)\,\phi_s\,+\,\ell_s\,(1-u_s)\,\psi_s\,dw\,ds 
	\end{align}
is a continuous martingale with quadratic variation
	\begin{equation}
	4\,\gamma \int_0^t  \int_\RR  u_s\,(1-u_s)\, \phi_s^2 +\ell_s\,(1-\ell_s)\,\psi_s^2\,+\,2\,\phi_s\,\psi_s\,\ell_s\,(1-u_s) \, dw \, ds.
	\end{equation}
Any sub-sequential limit $(u,\,\ell)$ solves this martingale problem. It is standard (see p. 536-537 in \cite{MT95}) to show that $(u,\,\ell)$ then solves the coupled SPDE in Theorem \ref{mvLimit2} weakly, with respect to some white noises.

\clearp 
	
\section{Green's function representation}

As remarked earlier our proof follows the approach in \cite{MT95}. The first step is to prove the analogue of their (2.11).
Observe that $u^n_t(z)= \<u^n_t,\,L\,{\bf 1}_z\>$ and $\ell^n_t(z)= \<\ell^n_t,\,L\,{\bf 1}_z\>$, where ${\bf 1}_z$ is the function on $L^{-1}\ZZ$ which is $1$ at $z$ and zero elsewhere. Let $\alpha_n= r_nM_nL_n^{-2}$ which converges to $\alpha $ as $n\to\infty$ and let 
\beq
p^{n}_t(w):=L\,\P(X^{n}_{t}=w\,|\,X^{n}_0=0)
\label{tprw}
\eeq
be the transition probability of the simple random walk $(X^{(n)}_t)_{t\geq 0}$ on $L^{-1}\ZZ$ with jump rate $2L^2$, so that it converges to $p_t(w)$
the transition density of Brownian motion run at rate 2.
Let $\{P^{n}_t\}_{t\geq 0}$ be the associated semigroup which has generator the discrete Laplacian $\Delta_L$ defined in \eqref{dLaplacian}.

Applying the approximate martingale problems with test function 
\begin{equation}\label{E:testfcn}
\phi_s(w):=\phi^{t,z}_s(w):= 
\begin{cases} p^{n}_{\alpha_n(t-s)}(w-z) & \text{for }s\in[0,t], \\ 0 & \hbox{otherwise} \end{cases}
\end{equation}
and using the facts that $\partial_s\phi_s+\alpha_n\Delta_{L}\phi_s=0$ and $\<u^n_0,\phi^{t,z}_0\> = P^n_{\alpha_nt}u^n_0(z)$, we 
have 
\beq
u^n_t(z) =   P^n_{\alpha_nt}u^n_0(z) + Y_t(\phi)+ Z_t(\phi)+E^{(1)}_t(\phi) + E^{(2)}_t(\phi) \label{E:Green_n}
\eeq
for $t\geq 0$ and $z\in L^{-1}\ZZ$.
Here $Z_t(\phi)$, $E^{(1)}_t(\phi)$ and $ E^{(2)}_t(\phi)$ are martingales defined in
\eqref{term1b}, \eqref{term3b} and \eqref{term2b} respectively, with $\phi_s$ defined in \eqref{E:testfcn}. To describe the other term, we let $\beta_n=M R^{-1}$
which converges to $\beta$, and let $Y_t(\phi)$ be 
\beq
\frac{\theta \beta_n}{L} \int_0^t \sum_{w\in L^{-1}\ZZ}
[u^n_s(w-L^{-1})+u^n_s(w+L^{-1})](1-u^n_s(w)) \phi_s(w) \,ds.
\label{defY}
\eeq

Repeating the last argument for the labeled particles, 
\beq
\ell^n_t(z) = P^n_{\alpha_nt}\ell^n_0(z) + Y^{\ell}_t(\phi) + Z^{\ell}_t(\phi)+E^{\ell,1}_t(\phi) + E^{\ell,2}_t(\phi). \label{E:Green_nell}
\eeq	
Here $Z^{(\ell)}_t(\phi):=Z^{0}_t(\phi)+Z^{2}_t(\phi)$ is given by \eqref{ltwoNoises}, $E^{(\ell,2)}_t(\phi)$ is obtained by replacing  $(\xi,\,\xi^c)$ by $(\eta,\,\eta^c)$ in \eqref{term3b}, $E^{(\ell,2)}_t(\phi)$ is the fluctuation term corresponding to \eqref{lterm2}. The remaining term is
\begin{align*} 
Y^{\ell}_t(\phi) = 	\frac{\theta \beta_n}{L}\int_0^t \sum_{w\in L^{-1}\ZZ}\Big[&\ell^n_{s}(w+L^{-1})-u^n_s(w+L^{-1})\ell^n_s(w)  \\
& + \ell^n_{s}(w-L^{-1})-u^n_s(w-L^{-1})\ell^n_s(w)\Big]\,\phi_s(w)\,ds.
\end{align*}

\clearp 

\section{Tightness} \label{sec:tight}

Recall that a sequence of probability measures is said to be $C$-tight, if it is tight in $D$ and  any subsequential limit has a continuous version. 
The goal of this section is to prove: 

\begin{theorem}\label{T:Tightuell}
Suppose the assumptions in Theorem \ref{mvLimit} hold. Then the sequence $\{(u^n,\,\ell^n)\}_{n\geq 1}$ is $C$-tight in  
$D([0,T],\,C_b(\RR)\times C_b(\RR))$ for every $T>0$. 
\end{theorem}
	
\begin{proof}	
The compact containment condition (condition (a) of Theorem 7.2 in \cite[Chapter 3]{EK86}) holds trivially as $0\leq \ell^n\leq u^n\leq 1$. Since $C_b(\RR)\times C_b(\RR)$ is equipped with the product metric, the desired $C$-tightness follows once we can show that for any $\epsilon>0$, one has
\begin{align}
& \lim_{\delta\to 0} \limsup_{n\to\infty}\,\P\,\bigg( \sup_{\substack{t_1-t_2<\delta\\0\leq t_2\leq t_1\leq T}}
\big\|u^n_{t_1}-u^n_{t_2} \big\|\,>\epsilon \bigg) = 0,
\label{E:Tightu}\\
& \lim_{\delta\to 0}\limsup_{n\to\infty}\,\P\,\bigg( \sup_{\substack{t_1-t_2<\delta\\0\leq t_2\leq t_1\leq T}}
\big\|\ell^n_{t_1}-\ell^n_{t_2} \big\|\,>\epsilon \bigg) = 0.
\label{E:Tightell}
\end{align}	
Here and in what follows the norm is the one defined in \eqref{normC}. 
It is enough to show that \eqref{E:Tightu} holds with $u^n$ replaced by any term in the decomposition given in \eqref{E:Green_n}, 
and that \eqref{E:Tightell} holds with $\ell^n$ replaced by any term in \eqref{E:Green_nell}.

\mn
{\bf First term in \eqref{E:Green_n} and \eqref{E:Green_nell}. } 
By standard coupling arguments for simple random walk, we can check as in Lemma 7(b) of \cite{MT95} that, upon linearly interpolating  
$ P^n_{\alpha_nt}u^n_0(z)$ in space, we have
\begin{equation}
\label{E:Semigroups}
\sup_{t\in[0,T]}\| P^n_{\alpha_nt}u^n_0 -P_{\alpha_n\,t}f_0\| \to 0 \quad\text{as }n\to\infty,
\end{equation}
where $\{P_t\}_{t\geq 0}$ is the semigroup for the Brownian motion in $\RR$ running at rate 2, and $f_0$ is the initial condition for $u$ functions in Theorem  \ref{mvLimit2}. This implies, by the continuity of the semigroup $P_t$, that \eqref{E:Tightu} holds with $u^n_t$ replaced by $P^n_{\alpha_nt}u^n_0$. By the same reasoning, we have
\begin{equation}\label{E:Semigroups_ell}
\sup_{t\in[0,T]}\| P^n_{\alpha_nt}\ell^n_0 -P_{\alpha_n\,t}g_0\| \to 0 \quad\text{as }n\to\infty 
\end{equation}
where $g_0$ is the initial condition for $\ell$ functions in Theorem  \ref{mvLimit2}.
		
\mn 
{\bf Remaining terms in \eqref{E:Green_n} and \eqref{E:Green_nell}. } For simplicity, we write 
\begin{align*}
\widehat{u}^n_t(z)&:=Y_t(\phi) + Z_t(\phi)+E^{(1)}_t(\phi) + E^{(2)}_t(\phi), \\
\widehat{\ell}^n_t(z)&:=  Y^{(\ell)}_t(\phi) + Z^{(\ell)}_t(\phi)+E^{(\ell,1)}_t(\phi) + E^{(\ell,2)}_t(\phi).
\end{align*}
The next moment estimate for space and time increments is similar to Lemma 6 in \cite{MT95}, but ours implies 
H\"older continuity of the limits with exponent $<1/2$ in space and $<1/4$ in time. 
\begin{lemma}
\label{L:Moment_hatu}
For any $p\geq 2$ and $T\geq 0$, there exists a constant $C(T,p)>0$  such that
\begin{align}
\label{E:Moment_hatu}
&\E |\widehat{u}^n_{t_1}(z_1)- \widehat{u}^n_{t_2}(z_2)|^p \leq C_{T,p}\,\Big(|t_1-t_2|^{p/4}+|z_1-z_2|^{p/2}+ M^{-p} \Big) \\
\label{E:Moment_hatell}
&\E |\widehat{\ell}^n_{t_1}(z_1)- \widehat{\ell}^n_{t_2}(z_2)|^p \leq C_{T,p}\,\Big(|t_1-t_2|^{p/4}+|z_1-z_2|^{p/2}+ M^{-p} \Big) \qquad
\end{align}
for all $0\leq t_2\leq t_1\leq T$, $z_1,\,z_2\in L^{-1}\ZZ$ and $n\geq 1$.
\end{lemma}

The proof of this result is postponed to the next section since it requires a number of computations. We now argue that \eqref{E:Moment_hatu} implies
\eqref{E:Tightu} holds for $\widehat{u}^n$. This idea is described in the paragraph before Lemma 7 in \cite{MT95} and page 648 of \cite{Kleim}: we approximate the c\`adl\`ag process $\widehat{u}^n$ by  a {\it continuous} process $\tilde{u}$ and invoke a tightness criterion inspired by Kolmogorov's continuity theorem. 
\begin{lemma}
\label{L:Moment_tildeu}
Define $\tilde{u}^n\in C([0,\infty),\,C_b(\RR))$ by $\tilde{u}_t=\widehat{u}_t$ on the grid $t\in \theta_n \,\ZZ_+$ and then linearly interpolate in $t$ for each $w\in\RR$. Suppose $\theta_n> M^{-4}$ and $\lim_{n\to\infty}\theta_n=0$. Then there exists $n_0\in\mathbb{N}$ such that for any $p\geq 2$, $T\geq 0$ and $K\geq 0$,
\begin{equation}\label{E:Moment_tildeu}
\E |\tilde{u}^n_{t_1}(z_1)- \tilde{u}^n_{t_2}(z_2)|^p \leq C_{T,p,K}\Big(|t_1-t_2|^{p/4}+|z_1-z_2|^{p/2} \Big),
\end{equation}
for all $0\leq t_2\leq t_1\leq T$, $z_i\in \RR$ with $|z_i|\leq K$ ($i=1,\,2$) and $n\geq n_0$. 
\end{lemma}
	 	
By a standard argument (see, for example, Problems 2.2.9 and 2.4.11 of \cite{KS91}), one can show that \eqref{E:Moment_tildeu} implies that \eqref{E:Tightu} holds when $u$ is replaced by $\tilde{u}^n$. Finally, by the reasoning in the proof of Lemma 7(a) in \cite{MT95}, there is a $\sigma>0$  so that 
$$
\limsup_{n\to\infty}\,\P\,\Big(\sup_{t\in[0,T]}\|\tilde{u}^n_{t}- \widehat{u}^n_{t}\| \geq n^{-\sigma}\Big)\,= 0.
$$
Therefore \eqref{E:Tightu} holds for $\widehat{u}^n$. By the same argument, \eqref{E:Moment_hatell} implies
\eqref{E:Tightell} holds for $\widehat{\ell}^n$. Theorem \ref{T:Tightuell} will be proved once Lemmas \ref{L:Moment_hatu} and \ref{L:Moment_tildeu} are. 	
\end{proof}

\clearp 

\section{Proofs of Lemmas  \ref{L:Moment_hatu} and \ref{L:Moment_tildeu}}

\begin{proof}[Proof of Lemma \ref{L:Moment_hatu}]
We prove only \eqref{E:Moment_hatu} for unlabeled particles. The proof of \eqref{E:Moment_hatell} for labeled particles is similar.
The basic ingredients are the following estimates of time and space increments of the 
transition probability of simple random walk. Namely, there exists a constant $C>0$  {\it independent of} $n$ such that
\begin{align}
&0\,\leq\, \int_0^{T} p^n_s(0)\,ds \leq  C\,\sqrt{T},
\label{E:LCLT0}\\
&0\,\leq\, \int_0^{\infty} p^n_s(0)-p^n_{s+\theta}(0)\,ds \leq  C\,\sqrt{\theta},
\label{E:LCLT1}\\
& 0\,\leq\,\int_0^{\infty} p^n_s(0)-p^n_s(z)\,ds \leq  C\,|z| 
\label{E:LCLT2}\\
&\int_0^{T} \frac{1}{L}\sum_{w}|p^n_{s+\theta}(w)-p^n_{s}(w)|\,ds \;\leq \; C\,\sqrt{T\,\theta},
\label{E:LCLT3}\\
&\int_0^{T} \frac{1}{L}\sum_{w} |p^n_{s}(w)-p^n_{s}(z+w)|\,ds\;\leq\;  C\,\sqrt{T}\,|z|
\label{E:LCLT4}
\end{align}
for $\theta\geq 0$, $z\in L^{-1}\ZZ$ and $T>0$. These estimates can be either found or deduced from the 
standard methods described in Chapter 2 in \cite{LL10}. For completeness, we give precise references
and missing details in the next section.

We will show that each of the four terms of $\widehat{u}^n$ satisfies \eqref{E:Moment_hatu}.
To simplify notation, we assume, without loss of generality for the proof, that $\alpha_n\equiv 1$. 
First, we deal with the process $Y$ that has no jumps. To reduce the size of the formulas we let
$$
v^n_s(w) = [u^n_s(w-L^{-1}) + u^n_s(w+L^{-1})] (1-u^n_s(w)).
$$
Using the definitions of $Y$ \eqref{defY} and of our test function \eqref{E:testfcn} we have for $t_1>t_2$
\begin{align*}
 Y_{t_1}(\phi^{t_1,z_1})& -  Y_{t_2}(\phi^{t_2,z_2}) =  \theta\beta_n \int_{t_2}^{t_1} 
\frac{1}{L} \sum_{w} v^n_s(w) p^n_{t_1-s}(z_1-w) \,ds \\
& + \theta\beta_n \int_0^{t_2} \frac{1}{L} \sum_{w} v^n_s(w) [p^n_{t_1-s}(z_1-w)-p^n_{t_2-s}(z_2-w)] \,ds \\
& \equiv \theta\beta_n(\Theta_1(Y) + \Theta_2(Y)).
\end{align*}
The sums are over $w\in L^{-1}\ZZ$, which we have omitted to simplify notation. 
Since $0\leq u^n\leq 1$ and $L^{-1}\sum_w p^n_t(z-w)=1$ (recall the definition of $p^n$ in \eqref{tprw}), we have
\beq
\E |\Theta_1(Y)|^p \leq 2^p\,(t_1-t_2)^p. \label{Y1}
\eeq	
by  H\"older's inequality.

By the triangle inequality and the translation invariance and symmetry of the transition density, we obtain
\begin{align*}
&\int_0^{t_2} \frac{1}{L}\sum_{w}  |p^n_{t_1-s}(z_1-w)-p^n_{t_2-s}(z_2-w)|\,ds \\
&\leq\, \int_0^{t_2} \frac{1}{L}\sum_{w} \big(|p^n_{t_1-s}(z_1-z_2+w) - p^n_{t_1-s}(w)|+|p^n_{t_1-s}(w)-p^n_{t_2-s}(w)|\big) \,ds\\
& \leq C_T|z_1-z_2| + C\sqrt{t_1-t_2}
\end{align*}		
by \eqref{E:LCLT4} and \eqref{E:LCLT3}. Since $0\leq v^n \leq 2$, 
\beq
\label{Y2}
\E\big[\,|\Theta_2(Y)|^p\,\big]\leq  C_{T,p}\, \big( \sqrt{t_1-t_2}+ |z_1-z_2| \big)^{p}
\eeq
for all $0\leq t_2\leq t_1\leq T$, $z_1,\,z_2\in L^{-1}\ZZ$ and $n\geq 1$.	
	
It remains to consider  $E^{(1)}_t(\phi)$, $ E^{(2)}_t(\phi)$ and $Z_t(\phi)$. Note that for each of them, the largest possible jump is bounded almost surely by 
$$
2(ML)^{-1}\sup_{s\geq 0} \|\phi_s\|_{\infty}\leq 2\,M^{-1}, 
$$ 
because $\phi_s$, as defined in \eqref{E:testfcn}, is bounded by $L$.

We shall employ a version of the Burkholder-Davis-Gundy inequality stated at the bottom of page 527 in \cite{MT95}. 
Namely, for any c\`adl\`ag martingale $X$ with $X_0=0$ and for $p\geq 2$,
\begin{equation}\label{E:BDG}
\E[\,\sup_{s\in[0,t]}|X_s|^p\,]\leq C(p)\,\E\Big[\,\<X\>_t^{p/2}+\sup_{s\in[0,t]}|X_s-X_{s-}|^p\,\Big],\quad t\geq 0. 
\end{equation}
Writing $N^{x,y}_s$ for the compensated Poisson process $P^{x,y}_s - r_ns$, 
and $\sum_{x, y\sim x}$ for the double sum $\sum_x \sum_{y\sim x}$, we decompose  
\begin{align*}
 E^{(1)}_{t_1}&(\phi^{t_1,z_1})  - E^{(1)}_{t_2}(\phi^{t_2,z_2})\\			
&=  \frac{1}{ML} \sum_{x, y\sim x} \int_{t_2}^{t_1} \xi_{s-}(y)\xi^c_{s-}(x) \,\big(p^{n}_{t_1-s}(z_1-y)-p^{n}_{t_1-s}(z_1-x) \big)\,dN^{x,y}_s\\
& +  \frac{1}{ML} \sum_{x, y\sim x} \int_{0}^{t_2}  \xi_{s-}(y)\xi^c_{s-}(x) \,\big(p^{n}_{t_1-s}(z_1-y)-p^{n}_{t_1-s}(z_1-x)\\
& \hphantom{MMM\sum_{x\sim y} \int_{0}^{t_2}  \xi_{s-}(y)\xi^c_{s-}(x)}
-p^{n}_{t_2-s}(z_2-y)+p^{n}_{t_2-s}(z_2-x) \big)\,dN^{x,y}_s\\
& \equiv\Theta_1(E^{(1)}) \,+\,\Theta_2(E^{(1)}).
\end{align*}
Writing $\tilde N^{x,y}_s$ for the compensated Poisson process $\tilde P^{x,y}_s - \theta R^{-1}_ns$ we have
\begin{align*}
 E^{(2)}_{t_1}&(\phi^{t_1,z_1}) - E^{(2)}_{t_2}(\phi^{t_2,z_2})\\			
&=  \frac{1}{ML} \sum_{x, y\sim x} \int_{t_2}^{t_1} \xi_{s-}(y)(1-\xi_{s-}(x)) \, p^{n}_{t_1-s}(z_1-x) \big)\,d\tilde N^{x,y}_s\\
 +  &\frac{1}{ML} \sum_{x, y\sim x} \int_{0}^{t_2}  \xi_{s-}(y)(1-\xi_{s-}(x)) [p^{n}_{t_1-s}(z_1-x)-p^{n}_{t_2-s}(z_2-x)] \,d\tilde N^{x,y}_s\\
& \equiv\Theta_1(E^{(2)}) \,+\,\Theta_2(E^{(2)}).
\end{align*}
Finally, writing $Q^{x,y}_s = N^{y,x}_s-N^{x,y}_s$ we have
\begin{align*}
Z_{t_1}&(\phi^{t_1,z_1}) - Z_{t_2}(\phi^{t_2,z_2})\\
&=  \frac{1}{ML} \sum_{x, y\sim x} \int_{t_2}^{t_1} (1-\xi_{s-}(y))\xi_{s-}(x) p^{n}_{t_1-s}(z_1-x) \,dQ^{x,y}_s\\
+ &\frac{1}{ML} \sum_{x, y\sim x} \int_{0}^{t_2} (1-\xi_{s-}(y))\xi_{s-}(x) [p^{n}_{t_1-s}(z_1-x)-p^{n}_{t_2-s}(z_2-x)]\,dQ^{x,y}_s\\
&\equiv \Theta_1(Z) \,+\,\Theta_2(Z).
\end{align*}

Once we use $0 \le \xi \le 1$ to simplify the integrands the three expressions have a similar structure. $E^{(1)}$ will be the smallest
since it has a difference of transition densities at adjacent sites, so we begin by estimating $E^{(2)}$.
To estimate $\Theta_1(E^{(2)})$, we let
$$
X^1_t = \frac{1}{ML} \sum_{x,y\sim x} \int_{0}^{t} \xi_{s-}(y)\xi^c_{s-}(x) p^{n}_{t_1-t_2-s}(z_1-x)\,d\tilde{N}^{x,y}_s.
$$
By Markov property of $(\xi_t)_{t\geq 0}$ and the stationarity of the compensated Poisson process, 
$$
\E(|\Theta_1(E^{(2)})|^p) =\E\,\E^{\xi_{t_2}}(|X_{t_1-t_2}|^{p}),
$$
where $\E^{\xi_{t_2}}$ is the expectation w.r.t.~the law of $\xi$ starting at $\xi_{t_2}$. 

To prepare for the next calculation we note that using the symmetry of the transition density and the Chapman-Kolmogorov equation
\beq
L^{-1}\sum_{w} [p^n_{u}(w)]^2  = L^{-1}\sum_{w} p^n_{u}(w)p^n_u(-w) = p_{2u}(0).
\label{CKtrick}
\eeq
The predictable bracket process of $X^1$ is
$$
\<X^1\>_t = \frac{\theta}{R_n (ML)^{2}} \sum_{x, y\sim x} \int_{0}^{t} \xi_{s-}(y)\xi^c_{s-}(x) [p^{n}_{t_1-t_2-s}(z_1-x)]^2\,ds.
$$
Since there are $2M$ values of $y$ for each $x$ and $M$ values of $x$ for each $w \in L^{-1}\ZZ$, if we let $c_n = 2\theta/R_nL_n$
then the above is  
\begin{align*}
&\leq c_n \int_0^{t} L^{-1}\sum_{w} \big(p^n_{t_1-t_2-s}(w)\big)^2\,ds
\\
&= c_n \int_0^{t} p^{n}_{2(t_1-t_2-s)}(0) \,ds \le c_n\,C\, (t_1-t_2)^{1/2}
\end{align*}
for $t\in[0,\,t_1-t_2]$,
where in the second step we have used the \eqref{CKtrick} and in the last step we used \eqref{E:LCLT0},
Note that $c_n\to 0$ since  $R_nL_n \to\infty$. 

If we do the calculation for $Z$ then $c_n = 2r_n/L \to 2\gamma$ so we get the same upper bound.
In $E^{(1)}$, $c_n = r_n/L_n \to \gamma$ but  
we have $p^{n}_{t_1-s}(z_1-y) -  p^{n}_{t_1-s}(z_1-x)$ instead of a single $p$, so 
\begin{align*}
L^{-1}\sum_{w} \big[p^n_{t_1-t_2-s}(w) - p^n_{t_1-t_2-s}(w-L^{-1})\big]^2& \\
 = 2p^n_{2(t_1-t_2-s)}(0) - 2p^n_{2(t_1-t_2-s)}(L^{-1}).&
\end{align*}
If we use \eqref{E:LCLT4} now we would get an upper bound of $C_TL^{-1}$ when we integrate the above expression with respect to $ds$ from $s=0$ to $s=T$.
Therefore, by \eqref{E:BDG}, we have
\beq
\E(|\Theta_1|^p )  \leq C_{p} ( |t_1-t_2|^{p/4} + M^{-p} ). \label{bd1}
\eeq
for $E^{(1)}$, $E^{(2)}$ and $Z$.

Similarly, $\E(|\Theta_2(E^{(2)})|^p) =\E(|X^2_{t_2}|^{p})$, where
$$
X^2_t = \frac{1}{ML} \sum_{x, y\sim x} \int_0^t  \xi_{s-}(y)\xi^c_{s-}(x) [p^{n}_{t_1-s}(z_1-x)-p^{n}_{t_2-s}(z_2-x)]\,d\tilde N^{x,y}_s\\
$$		
is a c\`adl\`ag martingale for $t\leq t_2$ with predictable bracket process
$$
\<X^2\>_t \leq  c_n \,\int_0^{t} L^{-1}\sum_{w} [p^{n}_{t_1-s}(z_1-w)-p^{n}_{t_2-s}(z_2-w)]^2.
$$
Arguing as before using \eqref{CKtrick} we get
$$
\leq c_n \int_0^{t} p^n_{2(t_1-s)}(0) + p^n_{2(t_2-s)}(0) - 2p^n_{t_2+t_1-2s}(z_1-z_2)
$$
a result that also holds for $E^{(1)}$ and $Z$.
Adding and subtracting $2p^n_{t_2+t_1-2s}(0)$ and using \eqref{E:LCLT1} and \eqref{E:LCLT2} the above is at most
$$
C \sqrt{t_1-t_2} + C_T |z_1-z_2|).
$$
Using \eqref{E:BDG} now, we have
\beq
\E(|\Theta_2|^p )  \leq C_{p,T} ( |t_1-t_2|^{p/4} + |z_1-z_2|^{p/2} + M^{-p} ) \label{bd1}
\eeq
which holds for $E^{(1)}$, $E^{(2)}$ and $Z$ and the proof is complete.
\end{proof}

\begin{proof}[Proof of Lemma \ref{L:Moment_tildeu}]
It suffices to consider the case $|t_1-t_2|^{1/4}  \leq M^{-1}$ and $|z_1-z_2|^{1/2}\leq M^{-1}$, since otherwise  Lemma \ref{L:Moment_hatu} easily implies \eqref{E:Moment_tildeu}. The triangle inequality gives
\beq
\label{E:Moment_tildeu2}
|\tilde{u}^n_{t_1}(z_1)- \tilde{u}^n_{t_2}(z_2)| \leq |\tilde{u}^n_{t_1}(z_1)- \tilde{u}^n_{t_2}(z_1)|+|\tilde{u}^n_{t_2}(z_1)- \tilde{u}^n_{t_2}(z_2)|.
\eeq
We first estimate the time difference on the right. Write $s_k:=k\theta_n$ for $k\in \ZZ_+$. Since $|t_1-t_2|\leq M^{-4}<\theta_n$, we have either $s_k\leq t_2<t_1 \leq s_{k+1}$ for some $k$ or $t_2<s_k<t_1$ for some $k$. Since $\tilde{u}^n$ is linear between grid points, in either case
$$
|\tilde{u}^n_{t_1}(z)- \tilde{u}^n_{t_2}(z)|
\leq 2\Big(|\widehat{u}^n_{s_{k+1}}(z)- \widehat{u}^n_{s_k}(z)| \vee |\widehat{u}^n_{s_k}(z)- \widehat{u}^n_{s_{k-1}}(z)|\Big)\,\frac{|t_1-t_2|}{\theta_n}.
$$
Hence Lemma \ref{L:Moment_hatu}, the assumption $M^{-4}<\theta_n$ and $|t_1-t_2|\leq M^{-4}$ imply that
\begin{align*}
\E |\tilde{u}^n_{t_1}(z)- \tilde{u}^n_{t_2}(z)|^p &\leq C_{T,p}\,\frac{|t_1-t_2|^p}{\theta_n^p}\big(\theta_n^{p/4}+M^{-p}\big) \\
&\leq C_{T,p}\,\Big(\frac{|t_1-t_2|}{\theta_n^{3/4}}\Big)^p \leq C_{T,p}\,|t_1-t_2|^{p/4}.
\end{align*}	 		
Next, we estimate the space difference on the right of \eqref{E:Moment_tildeu2}. Take $n$ large so that $M^{-1}< (1+\gamma) L^{-1}$ and $(1+\gamma)^{2} L^{-2}<L^{-1}$ (this is possible by our assumptions on scalings, even if $\gamma=0$). Then $|z_1-z_2|<(1+\gamma)^{2} L^{-2}< L^{-1}$. By almost the same argument used above, we easily obtain
\begin{align*}
\E |\tilde{u}^n_{t}(z_1) &- \tilde{u}^n_{t}(z_2)|^p \leq C(T,p)\,|z_1-z_2|^p\,L^p\big(L^{-p/2}+M^{-p}\big) \\
&\leq  C(T,p)\,|z_1-z_2|^p\,(L^{p/2}+(1+\gamma)^p)\\
&\leq C(T,p)\,\Big[\,(1+\gamma)^{p/2}\,|z_1-z_2|^{3p/4}+(1+\gamma)^p\,|z_1-z_2|^p\,\Big].
\end{align*}
The proof of Lemma \ref{L:Moment_tildeu} is complete.
\end{proof}

\clearp 

\section{Random walk estimates}

The first two, \eqref{E:LCLT0} and \eqref{E:LCLT1}, follow directly from the local central limit theorem (LCLT) (see, for example, 
Proposition 2.5.6 in \cite{LL10}). \eqref{E:LCLT3} follows from \eqref{E:LCLT4} and the Chapman-Kolmogorov equation: the integrand can be written as
$$
p^n_{s+\theta}(w)-p^n_{s}(w)=\frac{1}{L}\sum_{z\in L^{-1}\ZZ}p^n_{\theta}(z)\,\big(p^n_s(w-z)-p^n_s(w)\big).
$$
It remains to prove \eqref{E:LCLT2} and \eqref{E:LCLT4}.

By scaling, $p^n_t(w)=L\,p_{L^2t}(Lw)$ where $p_t(k)$ is the transition density of the simple random walk on $\ZZ$. The integral of \eqref{E:LCLT2} is therefore
$$\frac{1}{L}\int_0^{\infty} p^n_s(0)-p^n_s(Lz)\,ds.$$ 
Splitting this integral into two parts according to whether  $s\leq L|z|^2$ or $s > L|z|^2$, the first part is bounded by $L^{-1}\int_0^{L|z|^2} p^n_s(0)\,ds\leq  C|z|/\sqrt{L}$ according to \eqref{E:LCLT0}. The second part is bounded by $C|z|$ by the LCLT.

Formula \eqref{E:LCLT4} is similar to Proposition 2.4.1 in \cite{LL10} which says that 
\begin{equation}\label{2.4.1Lawer}
\sum_{k\in \ZZ}\big|q_m(k)-q_m(k+j)\big|\leq \frac{C\,j}{\sqrt{m}},
\end{equation}
where $q_m(k)$ is  the transition density for the discrete time simple random walk on $\ZZ$. Hence, by using scaling and an independent Poisson process $N_t$, we rewrite the left hand side of \eqref{E:LCLT4} as
\begin{align*}
&\frac{1}{L^2}\int_0^{L^2T}\sum_{k\in \ZZ} |p^n_{s}(k)-p^n_{s}(k+Lz)|\,ds\\
=&\,\frac{1}{L^2}\int_0^{L^2T}\sum_{k\in \ZZ} \big|\sum_{m=0}^{\infty}\P(N_s=m)\big(q_m(k)-q_m(k+Lz)\big)\big|\,ds
\end{align*}
which is at most
$$\frac{C\,|z|}{L}\int_0^{L^2T}\sum_{m=0}^{\infty}\frac{\P(N_s=m)}{\sqrt{m}}\,ds$$
by \eqref{2.4.1Lawer}. Arguing as in the proof of Proposition 2.5.6 in \cite{LL10} by using Proposition 2.5.5 (the LCLT for Poisson processes), we obtain that the integral is of order $\sqrt{L^2T}$ and hence \eqref{E:LCLT4} holds.

\clearp 

\section{Proof of Lemma \ref{SPDEunique}} \label{sec:unique}

To prepare for the proof for the SPDE, we begin by considering the diffusion process
\beq
dU  = \beta U(1-U) \, dt + \sigma \sqrt{U(1-U)} \, dB.
\eeq
Following the approach of Doering, Mueller and Smereka \cite{DMS}, we
change variables $Z=1-U$ to get (recall $dZ = -dU$)
\beq
dZ = -\beta Z(1-Z) \, dt - \sigma \sqrt{Z(1-Z)} \, dB.
\label{Zspde}
\eeq
The minus in front of the diffusion term is not important here but it will be in the next calculation in \eqref{jtSDE}.
Using It\^o's formula and ignoring the martingale terms
\begin{align*}
\hbox{drift}(Z^m) & = m Z^{m-1}(-\beta Z(1-Z))+ m(m-1) Z^{m-2}  \frac{\sigma^2}{2} Z(1-Z) \\
& = \beta m [ Z^{m+1} - Z^m ] + \frac{\sigma^2 m(m-1)}{2} [Z^{m-1} - Z^m] ).
\end{align*} 
Let $N(t)$ be a Markov process with $Q$ matrix
\begin{align}
& Q_{m,m+1} = \beta m \qquad Q_{m,m-1} = \sigma^2 \frac{m(m-1)}{2} 
\nonumber \\
&Q_{m,m} =  -\beta m - \sigma^2 \frac{m(m-1)}{2}. \label{Qmatrix}
\end{align}
Combining our calculations,
$$
\frac{d}{dt} EZ^m = \sum_n Q_{m,n} EZ^n
$$
Letting $P_{\ell,m}(t)=\P(N(t)=m\,|\,N(0)=\ell)$ be the transition probabilities, we have

\begin{lemma} \label{mart}
For fixed $T>0$ and $\ell \ge 1$, 
$M_t = \sum_{m=1}^\infty P_{\ell,m}(T-t)Z^m(t)$ is a martingale. 
\end{lemma}

\begin{proof}
Differentiating we have
\begin{align*}
\frac{d}{dt} EM_t & = \sum_m EZ^m(t) \frac{d}{dt} P_{\ell,m}(T-t) + P_{\ell,m}(T-t) \frac{d}{dt} EZ^m(t) \\
& = \sum_m  - EZ^m(t) \sum_n P_{\ell,n}(T-t) Q_{n,m} \\& + P_{\ell,m}(T-t) \sum_n Q_{m,n} E Z^n(t) =0
\end{align*}
if we interchange the roles of $m$ and $n$ in the second sum.
\end{proof} 

From Lemma \ref{mart} we get
\beq
EZ^{N(0)}(T) =EZ^{N(T)}(0).
\label{Zdual}
\eeq

Now consider the system
\begin{align}
dZ & = -\beta Z(1-Z) \, dt - \sigma\sqrt{VZ} \, dB^0 - \sigma\sqrt{Z(1-Z-V)} \, dB^1,
\label{jtSDE} \\
dV & = \beta VZ \, dt + \sigma\sqrt{VZ} \, dB^0 + \sigma\sqrt{V(1-V-Z)} \, dB^2,
\nonumber
\end{align}
where the $B^i$ are independent Brownian motions. 
To get our second dual function, we consider $Y_n = \sum_{m=1}^n Z^{m-1} V$. Using It\^o's formula
and for the second term recall \eqref{Zspde},
\begin{align*}
\hbox{drift}\left( \sum_{m=1}^n Z^{m-1} V \right) & = V \sum_{m=2}^n (m-1) Z^{m-2} (-\beta Z(1-Z)) \\
&+ \frac{\sigma^2}{2} V \sum_{m=3}^n (m-1)(m-2) Z^{m-3} Z(1-Z) \\
&+ \sum_{m=1}^n Z^{m-1} \beta ZV - \sigma^2 \sum_{m=2}^n (m-1) Z^{m-2} VZ,
\end{align*}
where the last term comes from the fact that the covariance of $Z$ and $V$ is $-\sigma^2 VZ$.
Collecting the terms with $\beta$, and changing variables $k=m-1$ in the second sum  we get
$$
= \beta V \left[ \sum_{m=2}^n (m-1) Z^m - \sum_{k=1}^{n-1} k Z^k + \sum_{m=1}^n Z^m \right].
$$
Adding the third sum to the first
\begin{align*}
&= \beta V \left[ \sum_{m=1}^n m Z^m - \sum_{k=1}^{n-1} k Z^k \right] \\
&= n\beta V Z^n = n\beta \left( \sum_{m=1}^{n+1} Z^{m-1} V - \sum_{m=1}^{n} Z^{m-1} V \right)
\end{align*} 
which corresponds to jumps from $n$ to $n+1$ at rate $\beta n$.
Collecting the terms with $\sigma^2$ and changing variables to have $Z^{k-1}$, we get
$$
= \sigma^2 V \left[ \sum_{k=2}^{n-1} \frac{k(k-1)}{2} Z^{k-1} - \sum_{k=3}^n \frac{(k-1)(k-2)}{2} Z^{k-1} - \sum_{k=2}^n (k-1) Z^{k-1} \right].
$$
Moving terms $k=2$ to $n-1$ from the last sum into the first one
$$
= \sigma^2 V \left[ \sum_{k=2}^{n-1} \frac{(k-2)(k-1)}{2} Z^{k-1} - \sum_{k=3}^n \frac{(k-1)(k-2)}{2} Z^{k-1} - (n-1) Z^{n-1} \right].
$$
The $k=2$ term in the first sum vanishes. Terms 3 to $n-1$ in the first sum cancel with 
the second sum leaving
$$
= - \sigma^2 V \frac{n(n-1)}{2} Z^{n-1} = \sigma^2 \frac{n(n-1)}{2} \left( \sum_{m=1}^{n-1} Z^{m-1} V - \sum_{m=1}^{n} Z^{m-1} V \right)
$$
which corresponds to jumps $n$ to $n-1$ at rate $\sigma^2 n(n-1)/2$.

Combining our calculations
$$
\frac{d}{dt} EY_n = \sum_m Q_{m,n} EY_n.
$$
where $Q_{m,n}$ is given in \eqref{Qmatrix}. Using Lemma \ref{mart} again,
$$
E\left( \sum_{m=1}^{N(0)} Z^{m-1}(T) V(T) \right) =  E\left( \sum_{m=1}^{N(T)} Z^{m-1}(0) V(0) \right).
$$

\subsection{Duality for the Wright-Fisher SPDE}

We begin by proving the duality result for the equation for single SPDE \eqref{uSPDE}. This is a known result due to Shiga \cite{Shiga88}.
However, he did not give many details and we need to generalize his result to our coupled SPDE, 
so we will follow the approach of Athreya and Tribe  \cite{AthTri}.  
Let $z = 1-u$. Define $\bar z_t(x) = \int z_t(y) p^\ep(y-x) \, dy$ where $p^\ep$ is the normal density with mean 0
and variance $\ep$. Noting that $x \to \bar z_t(x)$ is smooth and
using the weak formulation of \eqref{uSPDE} with test function $\phi^{\ep,x}(y)= p^\ep(y-x)$, we have
\begin{align}
\bar z_t(x) - \bar z_0(x) & = \int_0^t \alpha \Delta \bar z_s(x) \, ds 
\nonumber\\
& - \theta\beta \int_0^t \int 2z_s(y)(1-z_s(y)) p^\ep(y-x) \, dy \,ds
\label{zmart}\\
& + \int_0^t \int \sqrt{4\gamma z_s(y)(1-z_s(y)) } p^\ep(y-x) \, dW.
\nonumber
\end{align}

Using It\^o's formula (each $\bar z_t(x)$ is a semi-martingale so this is legitimate)
and writing $\mathcal{L}_z$ for the generator of $(\bar z_t(x_1), \ldots, \bar z(x_n))$  with $x_1, \ldots x_n$ fixed, we see that (ignoring the martingale terms)
\begin{align}
&\hbox{drift}\left( \mathcal{L}_{\bar z} \prod_i \bar z_t(x_i) \right) = 
\sum_{i=1}^n \prod_{j \neq i} \bar z_t(x_j) \alpha \Delta \bar z_t(x_i) 
\nonumber\\
& \qquad +  2\theta\beta \sum_{i=1}^n \prod_{j \neq i} \bar z_t(x_j) \int [z_t^2(y) - z_t(y) ] p^\ep(y-x_i) \, dy 
\label{zdual}\\
& \qquad + 4\gamma \sum_{i=1}^{n-1} \sum_{j=i+1}^n \prod_{k \neq i,j} \bar z_t(x_k) 
\int [ z_t(y)(1 - z_t(y)) ] p^\ep(y-x_i) p^\ep(y-x_j) \, dy.
\nonumber
\end{align}
The dual process is a system of branching coalescing Brownian particles. During their lifetime the
particles are Brownian motions run at rate $2\alpha$ with each giving birth at rate $2\theta\beta$.
In addition, for $i < j$, particle $j$ is killed by particle $i$ at rate $4\gamma L^{i,j}_t$ where 
$L^{i,j}_t$ denotes the local time of the process $x_j-x_i$ at $0$.
Writing $\mathcal{L}_x$ for the generator, we have 
\begin{align}
&\hbox{drift}\left(\mathcal{L}_x \prod_i \bar z_t(x_i) \right) = 
\sum_{i=1}^n \prod_{j \neq i} \bar z_t(x_j) \alpha \Delta \bar z_t(x_i) 
\nonumber\\
& \qquad + 2\theta\beta\sum_{i=1}^n  \prod_{j \neq i} \bar z_t(x_j) \cdot [\bar z_t^2(x_i) - \bar z_t(x_i) ] 
\label{zdual2} \\
& \qquad + 4\gamma \sum_{i=1}^{n-1} \sum_{j= i+1}^n  \prod_{k \neq i,j} \bar z_t(x_k) \cdot 
[\bar z_t(x_i)(1- \bar z_t(x_j)) ]\,\delta_{\{x_j=x_i\}} \nonumber
\end{align}
in which we used the formal notation $dL^{i,j}_t=\delta_{\{x_j(t)=x_i(t)\}}\,dt$. The precise meaning of 
the last term involves integration w.r.t. local times and is explained in \eqref{3rdLT0}.

We now follow Proposition 1 in \cite{AthTri} to use the duality method encapsulated in Theorem 4.4.11 of Ethier and Kurtz \cite{EK86}. In their notation
$\alpha=\beta=0$.
$$
F(\bar z, x) = \prod_{i=1}^{n} \bar z(x_i) \quad\text{if }x=(x_1,\cdots,\,x_n).
$$
They suppose
$
F(\bar z_t, x ) - \int_0^t G(\bar z_s, x) \, ds\,$ and $\,F(\bar z, x(t) ) - \int_0^t H(\bar z, x(s)) \, ds
$
are martingales and conclude that for $t\geq 0$,
\begin{equation}\label{EKdual}
	\E F(\bar z_t, x(0)) - \E F(\bar z_0, x(t)) = \E \int_0^t G(\bar z_{t-s}, x(s)) - H(\bar z_{t-s}, x(s)) \, ds.
\end{equation}
In our situation, \eqref{EKdual} holds with  $G(\bar z, x)=\mathcal{L}_{\bar z}F(\bar z, x)$ and $H(\bar z, x)=\mathcal{L}_{x}F(\bar z, x)$.

By the continuity of $x\mapsto z_t(x)$, we have $F(\bar z_t, x(0)) \to F(z_t, x(0))$ and  $F(\bar z_0, x(t)) \to F(z_0, x(t))$ a.s. as $\ep\to 0$,
so using the bounded convergence theorem, 
$$
\E F(\bar z_t, x(0)) - \E F(\bar z_0, x(t)) \to  \E F(z_t, x(0)) - \E F(z_0, x(t)).
$$
To prove the desired duality formula
\begin{equation}\label{WFdual}
\E \prod_{i=1}^{n(0)} z_t(x_i(0)) =  \E \prod_{i=1}^{n(t)} z_0(x_i(t)),\quad t\geq 0,
\end{equation}
it remains to argue that the RHS of \eqref{EKdual} tends to zero as $\epsilon\to 0$.

The first term in \eqref{zdual} agrees with that of \eqref{zdual2}. For the second terms, omitting $2\theta\beta$, the integrand of the difference is
\beq
\E\int_0^t\sum_{i=1}^{n(s)} \prod_{j \neq i} \bar z_{t-s}(x_j(s)) \Big[\int z_{t-s}^2(y) p^\ep(y-x_i(s))\,dy-\bar z_{t-s}^2(x_i(s))\Big] \,ds \to 0
\label{2ndLT}
\eeq
a.s. for $s\in (0,t)$, by continuity of $y\mapsto z_s(y)$ and dominated convergence. The contribution to \eqref{EKdual} from the third term of \eqref{zdual2} is (omitting $4\gamma$)
\begin{equation}\label{3rdLT0}
\E\int_0^t\sum_{i=1}^{n(s)-1} \sum_{j= i+1}^{n(s)}  \prod_{k \neq i,j}  \bar z_{t-s}(x_k(s)) \cdot 
[\bar z_{t-s}(x_i(s))\big(1-  \bar z_{t-s}(x_j(s))\big) ]\,dL^{i,j}_s
\end{equation}
which converges, by dominated convergence to
\begin{equation}\label{3rdLT}
\E\int_0^t\sum_{i=1}^{n(s)-1} \sum_{j= i+1}^{n(s)}  \prod_{k \neq i,j}  z_{t-s}(x_k(s)) \cdot 
[ z_{t-s}(x_i(s))\big(1-  z_{t-s}(x_j(s))\big) ]\,dL^{i,j}_s.
\end{equation}
Finally, we consider the contribution to \eqref{EKdual} from the third term of \eqref{zdual}. 
After the substitution $y\mapsto y+ x_i(s)$, we have (omitting $4\gamma$)
\begin{equation}\label{3wY}
\E \int_0^t \sum_{i=1}^{n(s)-1} \sum_{j=i+1}^{n(s)} 
\int  p^\ep(y) p^\ep(y+x_i(s)-x_j(s))\,Y^{i,j}_{s,t}(y) \, dy \,ds,
\end{equation}
where for any $i,\,j$,
$$
Y^{i,j}_{s,t}(y)= \prod_{k \neq i,j} \bar z_{t-s}(x_k(s))\cdot [ z_{t-s}(y+ x_i(s))(1 - z_{t-s}(y+ x_i(s))) ].
$$
At this point we would like to apply Lemma 2 of \cite{AthTri},  to obtain 
\begin{equation}\label{Occup}
\int_0^t p^\ep(y+x_i(s)-x_j(s))\,Y^{i,j}_{s,t}(y)\,ds = \int\int_0^tp^\ep(y+z)\,Y^{i,j}_{s,t}(y)\,dL^{i,j,z}_{s}\,dz,
\end{equation}
where  $L^{i,j,z}_{t}$ denotes the local time of the process $x_j-x_i$ at $z$. Their formula asssumes $Y$ is predictable,
so we substitute $s-$ for $s$ and note that this does not change the integral in \eqref{3wY}.

Putting \eqref{Occup} into \eqref{3wY}, then using the continuity of the local time and that of $Y^{i,j}_{s,t}(y)$ (see details in pages 1724--1725 of \cite{AthTri}), the integrand of \eqref{3wY} converges a.s.\ to the integrand of \eqref{3rdLT}. Convergence of expectations then follows from dominated convergence and the proof of \eqref{WFdual} is complete.

\subsection{Duality for the coupled SPDE}

To prove the duality formula for the coupled SPDE, we order the particles in $x_i(t)$, $i \le n(t)$ in such a way that (i) when two particles coalesce we keep the smaller index and (ii) immediately after particle $i$ gives birth, its offspring has index $i$ and all particles with index $\ge i$ (including the one that just gave birth) increase their index by 1. Adding the number of
particles in the dual as another variable to help clarify things,  for $k\geq 1$ we let
 \begin{equation}
 F_k( (z,\ell), (x,n) ) =\sum_{1\leq j_1<j_2<\cdots<j_k\leq n}\ell(x_{j_k})\cdots\ell(x_{j_1})\prod_{i=1}^{j_1-1} z(x_i).
 \end{equation}
The following duality formula is motivated by \eqref{ldual}: 
\begin{equation}\label{Coupledual}
\E F_k((z_t,\ell_t),(x_0,n_0)) = \E F_k((z_0,\ell_0),(x_t,n_t)), \quad k\geq 1.
\end{equation}

Once we have shown \eqref{Coupledual}, the uniqueness claimed in Lemma \ref{SPDEunique} follows, since it allows us to 
conclude that
the distribution at a fixed time is unique by identifying all cross moments of the bounded functions $u$ and $\ell$.  Uniqueness of the law of the process $(u,\ell)$ then follows from Theorem 4.2 in Chapter 4 of Ethier and Kurtz \cite{EK86}. The duality function \eqref{WFdual} for the Wright-Fisher SPDE can be viewed as the case $k=0$.  We shall provide the details of proof of \eqref{Coupledual} for the case $k=1$, the other cases are similar.

We let $\bar \ell_t(x) = \int \ell_t(y) p^\ep(y-x) \, dy$ and note
\begin{align*}
\bar z_t(x) - \bar z_0(x) & = \int_0^t \alpha \Delta \bar z_s(x) \, ds \\
& - 2\theta\beta \int_0^t \int z_s(y)(1-z_s(y)) p^\ep(y-x) \, dy\,ds \\
& - \int_0^t \int \sqrt{4\gamma z_s(y)\ell_s(y) }\, p^\ep(y-x) dW^0 \\
& - \int_0^t \int \sqrt{4\gamma z_s(y)(1-z_s(y)-\ell_s(y)) }\, p^\ep(y-x) dW \\
\bar \ell_t(x) - \bar \ell_0(x) & = \int_0^t \alpha \Delta \bar \ell_s(x) \, ds \\
& + 2\theta\beta \int_0^t \int z_s(y)\ell_s(y) p^\ep(y-x) \, dy\,ds \\
& + \int_0^t \int \sqrt{4\gamma z_s(y)\ell_s(y) }\, p^\ep(y-x) dW^0 \\
& + \int_0^t \int \sqrt{4\gamma \ell_s(y)(1-z_s(y)-\ell_s(y)) }\, p^\ep(y-x) dW^2.
\end{align*}
Writing $\mathcal{L}_{\bar z,\bar\ell}$ for the generator of $(\bar z_t(x_1), \ldots, \bar z(x_{k-1}), \bar\ell(x_k))$  
with $x_1, \ldots x_n$ and $n$ fixed, {drift}$[\mathcal{L}_{\bar z,\bar\ell} F_1( (\bar z_t,\bar \ell_t), (x,n) )]$
\begin{align}
&= \sum_{k=2}^n \sum_{i=1}^{k-1} \prod_{1\le j \le k \atop j \neq i} \bar z_t(x_j)\cdot \alpha \Delta \bar z_t(x_i)\cdot \bar\ell_t(x_k)
+  \prod_{j=1}^{k-1} \bar z_t(x_j) \alpha \Delta \bar\ell_t(x_k) \label{f1}\\
& + \sum_{k=2}^n 2\theta\beta \sum_{i=1}^{k-1} \prod_{j \neq i} \bar z_t(x_j)\cdot \bar\ell_t(x_k)  \int [z_t^2(y) - z_t(y) ] p^\ep(y-x_i) \, dy 
\label{f2}\\
& + \sum_{k=1}^n 2\theta\beta \prod_{j=1}^{k-1} \bar z_t(x_j)\cdot \int  z_t(y) \ell_t(y) p^\ep(y-x_k) \, dy 
\label{f3}\\
& + \sum_{k=3}^n 4\gamma \sum_{j=2}^{k-1} \sum_{i=1}^{j-1} \prod_{1\le h \le k-1 \atop h \neq i,j} \bar z_t(x_h) \cdot \bar\ell_t(x_k) 
\label{f4} \\
&\hphantom{MMMMMMMMM}
\int [ z_t(y)(1 - z_t(y)) ] p^\ep(y-x_i) p^\ep(y-x_j) \, dy  \nonumber\\
& - \sum_{k=2}^n 4\gamma \sum_{i=1}^{k-1}  \prod_{1\le h \le k-1 \atop h \neq i} \bar z_t(x_h) 
\int [ z_t(y)\ell_t(y) ] p^\ep(y-x_i) p^\ep(y-x_k) \, dy. \label{f5}
\end{align}

Writing $\mathcal{L}_{x,n}$ for the generator of the dual ordered  particle system, we want to compute
\begin{equation}\label{driftDual}
\hbox{drift} [ \mathcal{L}_{x,n} F_1( (\bar z,\bar \ell), (x,n)) ].
\end{equation}

and to show, as in \eqref{EKdual}, that
\beq
\E \int_0^t \mathcal{L}_{\bar z,\bar\ell} F_1( (\bar z_{t-s},\bar \ell_{t-s}), (x(s),n(s))) - \mathcal{L}_{x,n}  F_1( (\bar z_{t-s},\bar \ell_{t-s}), (x(s),n(s))) \, ds \to 0.
\label{goal}
\eeq
The terms coming from particle Brownian motions are
\beq
\sum_{k=1}^n \left( 
\sum_{i=1}^{k-1} \prod_{1\le j \le k-1 \atop j \neq i} \bar z_t(x_j) \alpha \Delta \bar z_t(x_i) \cdot \bar\ell_t(x_k) 
+ \prod_{i=1}^{k-1} \bar z_t(x_i) \cdot \Delta \bar\ell_t(x_k)\right)
\label{llap}
\eeq
which agree with \eqref{f1}, and hence cancel in \eqref{goal}.

{\bf Birth terms.} Given a vector $x = (x_1, \ldots x_n)$, let $x^i = (x_1, \ldots x_i, x_i, \ldots x_n)$ be the vector of length $n+1$
with the $i$ coordinate duplicated.
The total change in the drift \eqref{driftDual} due to births is (omitting the $2\theta\beta$)
\beq
\sum_{i=1}^n \left[ \sum_{k=i+1}^{n+1} \prod_{j=1}^{k-1} \bar z_t(x^i_j) \ell_t(x^i_k) 
-  \sum_{k=i+1}^{n} \prod_{j=1}^{k-1} \bar z_t(x_j) \ell_t(x_k) \right].
\label{dupe3}
\eeq
Here $i$ is the location of the duplication and there is no change in the drift for terms with $k\le i$. The difference
\beq
\sum_{k=i+2}^{n+1} - \sum_{k=i+1}^n = [ \bar z_t(x_i) - 1 ]  \sum_{k=i+1}^{n} \prod_{j=1}^{k-1} \bar z_t(x_j) \ell_t(x_k)
\label{Lxd1}
\eeq
since when $k \ge i+2$, we have $x^i_{k+1} = x_k$ and $\bar z_t(x_i)$ appears twice in the product. After we interchange the order of the summation this 
agrees with \eqref{f2} if we replace $p^\ep(y-x_i)$ in that formula by a pointmass at $x_i$. The remaining term ($k=i+1$) in the first sum in \eqref{dupe3} is
\beq
\sum_{i=1}^n \prod_{j=1}^{i} \bar z_t(x_j) \ell_t(x_i)
\label{Lxd2}
\eeq
since $x^i_i = x^i_{i+1} = x_i$. This agrees with \eqref{f3} if we again replace $p^\ep(y-x_i)$ by a pointmass at $x_i$. 
As we argued in \eqref{2ndLT} it follows that
$$
\E \int_0^t  \eqref{f2} - \eqref{Lxd1} + \eqref{f3} - \eqref{Lxd2} \, ds \to 0
$$

\ms
{\bf Killing terms.} Given a vector $x = (x_1, \ldots x_n)$ let $\hat x^j = (x_1, \ldots x_{j-1}, x_{j+1}, \ldots x_n)$ be the vector of length $n-1$
with the $j$ coordinate removed. The total change in the drift \eqref{driftDual} due to deaths is (omitting the $4\gamma$)
\beq
\sum_{j=2}^{n} \sum_{i=1}^{j-1} \left[ \sum_{k=j}^{n-1} \prod_{h=1}^{k-1} \bar z_t(\hat x^j_h) \ell_t(\hat x^j_k) 
-  \sum_{k=j}^{n} \prod_{h=1}^{k-1} \bar z_t(x_h) \ell_t(x_k) \right]\,\delta_{\{x_j=x_i\}}.
\label{kill3}
\eeq
Here $j$ is the location of the deletion and there is no change in the drift for terms with $k< j$. When $j \le k$, $\hat x^j_k = x_{k+1}$
so we have
$$
\sum_{k=j}^{n-1} - \sum_{k=j+1}^n 
= \sum_{k=j+1}^n \prod_{1\le h \le k-1 \atop h \neq j} \bar z_t(x_h) [ 1 - \bar z_t(x_j) ] \bar \ell_t(x_k).
$$

The remaining term coming from $k=j$ in the second sum is
$$
- \prod_{h=1}^{j-1} \bar z_t(x_h)  \bar\ell_t(x_j).
$$
Using these results in \eqref{kill3} and noting that in the first case $j=n$ is impossible, we have
\begin{align}
&\sum_{j=2}^{n-1} \sum_{i=1}^{j-1} \sum_{k=j+1}^n \prod_{1\le h \le k-1 \atop h \neq i,j} 
\bar z_t(x_h) \cdot \bar \ell_t(x_k) \cdot \bar z_t(x_i)[ 1 - \bar z_t(x_j) ]  \,\delta_{\{x_j=x_i\}}
\label{kill4}\\
&- \sum_{j=2}^{n} \sum_{i=1}^{j-1} \prod_{1 \le h \le j-1, h \neq i} \bar z_t(x_h) \cdot \bar z_t(x_i) \bar\ell_t(x_j) \,\delta_{\{x_j=x_i\}}.
\nonumber
\end{align}
where in the second sum we have changed $j$ to $k$, and in both terms we have removed an additional term from the product over $h$. On other hand,
interchanging the order of summation in \eqref{f4}
$$
\sum_{k=3}^n \sum_{j=2}^{k-1} \sum_{i=1}^{j-1}  = \sum_{j=2}^{n-1} \sum_{k=j+1}^n \sum_{i=1}^{j-1} = \sum_{j=2}^{n-1} \sum_{i=1}^{j-1} \sum_{k=j+1}^n 
$$
Hence formulas \eqref{f4} and \eqref{f5} become 
\begin{align}
\sum_{j=2}^{n-1} \sum_{i=1}^{j-1} \sum_{k=j+1}^n \prod_{1\le h \le k-1 \atop h \neq i,j}&  \bar z_t(x_h) \cdot \bar\ell_t(x_k) 
\nonumber\\
& \int [ z_t(y)(1 - z_t(y)) ] p^\ep(y-x_i) p^\ep(y-x_j) \, dy\label{f45}  \\
 - \sum_{k=2}^n 4\gamma \sum_{i=1}^{k-1}  \prod_{1\le h \le k-1 \atop h \neq i}& \bar z_t(x_h) 
\int [ z_t(y)\ell_t(y) ] p^\ep(y-x_i) p^\ep(y-x_k) \, dy. 
\nonumber
\end{align}
Arguing as in \eqref{3rdLT0}--\eqref{Occup} now completes the proof of \eqref{goal}. The proof of  \eqref{Coupledual} is complete, and hence that of Lemma \ref{SPDEunique}.

\clearp

\end{document}